\documentclass[10t,a4paper]{amsart}

\usepackage[english]{babel}
\usepackage[T1]{fontenc}
\usepackage[utf8]{inputenc}

\usepackage{amscd,amssymb,amsopn,amsmath,amsthm,graphics,amsfonts,enumerate,verbatim,calc}
\usepackage{amssymb, tikz}
\usepackage{mathtools}

\usepackage{hyperref}
\usepackage[capitalise]{cleveref}

\usepackage{setspace}
\usepackage{mathpazo}
\usepackage{color}
\usepackage{latexsym}
\usepackage{amsthm,amsfonts,amssymb,mathrsfs}
\usepackage{rotating}
\usepackage[leqno]{amsmath}
\usepackage{xspace}
\usepackage[all]{xy}
\usepackage{longtable}
\headsep=5 mm \numberwithin{equation}{section}
\newtheorem{theorem}{Theorem}[section]
\newtheorem{lemma}[theorem]{Lemma}
\newtheorem{notation}[theorem]{Notation}
\newtheorem{proposition}[theorem]{Proposition}
\newtheorem{corollary}[theorem]{Corollary}

\theoremstyle{definition}
\newtheorem{definition}[theorem]{Definition}
\theoremstyle{remark}
\newtheorem{remark}[theorem]{Remark}
\newtheorem{fact}[theorem]{Fact}
\newtheorem{example}[theorem]{Example}
\newtheorem{observation}[theorem]{Observation}
\newtheorem{discussion}[theorem]{Discussion}
\newtheorem{question}[theorem]{Question}

\newcommand{\Ass}{\operatorname{Ass}}
\newcommand{\im}{\operatorname{im}}

\newcommand{\grade}{\operatorname{grade}}

\newcommand{\tor}{\operatorname{tor}}
\newcommand{\Spec}{\operatorname{Spec}}

\newcommand{\Tr}{\operatorname{Tr}}

\newcommand{\Ht}{\operatorname{ht}}
\newcommand{\pd}{\operatorname{pd}}

\newcommand{\Syz}{\operatorname{Syz}}

\newcommand{\Der}{\operatorname{Der}}

\newcommand{\fd}{\operatorname{fl.dim}}

\newcommand{\V}{\operatorname{V}}

\newcommand{\id}{\operatorname{id}}
\newcommand{\Ext}{\operatorname{Ext}}

\DeclareMathOperator{\mz}{\mathcal{Z}}

\newcommand{\Supp}{\operatorname{Supp}}\newcommand{\nCM}{\operatorname{nCM}}
\newcommand{\eSupp}{\operatorname{hSupp}}
\newcommand{\Tor}{\operatorname{Tor}}

\newcommand{\Hom}{\operatorname{Hom}}

\newcommand{\End}{\operatorname{End}}

\newcommand{\depth}{\operatorname{depth}}

\newcommand{\coker}{\operatorname{coker}}

\newcommand{\vpl}{\operatornamewithlimits{\varprojlim}}

\newcommand{\vil}{\operatornamewithlimits{\varinjlim}}

\newcommand{\lo}{\longrightarrow}
\newcommand{\fm}{\frak{m}}
\newcommand{\fp}{\frak{p}}
\newcommand{\fq}{\frak{q}}
\newcommand{\fa}{\frak{a}}

\newcommand{\HH}{H}

\newcommand{\NN}{\mathbb{N}}

\begin{document}

\author[]{Mohsen Asgharzadeh}

\address{}
\email{mohsenasgharzadeh@gmail.com}

\title[ ]
{Cohomological splitting, realization, and  finiteness}


\begin{abstract} 
	
We search for some splitting (resp. finiteness) criteria of a given module $M$ over a local ring $(R,\fm,k)$ in terms
of the splitting (resp. finiteness) property of certain 
cohomological functors evaluated at $M$, for example the splitting of $H^i_\fm(M)$. In particular, we deal with the cohomological
splitting question posted by Vasconcelos.    We present a connection from our approach 
to the realization problem of Nunke. This is equipped with several applications. 
For instance, we recover some results of Jensen (and others) by applying simple methods. Additional
applications, including a  computation of the projective dimension of some injective modules,
are given. This enables us to extend some results of Matlis  (resp. Osofsky) on the projective dimension of $E_R(k)$ (resp. $\mathcal{Q}$).

\end{abstract}

\maketitle

\setcounter{tocdepth}{1}
\tableofcontents

\section{Introduction}

A short exact sequence  $\zeta:= 0\to M_1\to M\to M_2\to 0$ of modules over a local ring $(R,\fm)$ is called cohomologically splits at level $i$ provided the $i$-th local cohomology module $\HH^i_{\fm}(M)$
decomposes into $\HH^i_{\fm}(M_1)\oplus\HH^i_{\fm}(M_2)$. Here, we study the cohomological
splitting property and its connection with the classical splitting property of $\zeta$. 
The initial motivation comes from the following beautiful theorem of Vasconcelos \cite{v}:

\begin{theorem}\label{1.1}
	Let $(R,\fm)$ be a $1$-dimensional Gorenstein local ring and let  $\zeta$ be the exact sequence $ 0\to M_1\to M \to M_2\to 0$   of finitely generated modules of projective dimension at most one. If $$\tor(M)=\tor(M_1)\oplus\tor(M_2),$$ then $\zeta$ splits.
\end{theorem}

For each  functor $\mathcal{F}$, we assign the concept of  $\mathcal{F}$-splitting. In \S2 we deal with  the tensor and hom functors, and study the corresponding splitting types.
Then we investigate splitting with respect to their derived functors, i.e., we study  the $\Tor$-splitting and $\Ext$-splitting.  These continue the work of  Guralnick \cite{Gur}, and may regard as a root of cohomological
splitting property.

 Vasconcelos posted the following splitting question:
  \begin{question}\label{vq}
Does Theorem \ref{1.1} hold  for  any $1$-dimensional Cohen-Macaulay   rings? 
\end{question}

  \S3 equipped
with a series of observations about Question \ref{vq} and presents a higher-dimensional analogue of Vasconcelos' theorem. We do these, by applying various aspects of cohomological
splitting property. For  samples, see Proposition \ref{o3.c} and \ref{p1}. Also, Theorem \ref{high} and its corollaries deal with  Question \ref{vq} . 

The next goal is to find  sub-functors  or even direct summands of  $\HH^i_{\fm}(-)$, see Proposition \ref{ptri} and subsequent examples. This is inspired  from Auslander's comments on the functor ext \cite{comment}, where
he investigated sub-functors of $\Ext$. Using these, we select the following  application:

\begin{proposition}\label{1.41}
	Let $(R,\fm)$ be a  complete  Gorenstein  local ring and $M$ be finitely generated. Suppose $\HH^i_{\fm}(M)$
is nonzero and injective. Then	$i=\dim R$.	 In particular,  $M$  is  $(S_2)$ if and only if $M$ is free.	
\end{proposition}

We note that freeness of a reflexive module is a challenging problem. For instance, see \cite{moh2} and references therein. 
Despite  the importance of  $(S_2)$, the following  easy consequence of Proposition \ref{1.41} holds even  the module is not assumed $(S_2)$.\begin{observation}\label{1.4}
Let $R:=k[[x_1,\ldots,x_d]]$ where $k$ is a field of zero characteristic  and let $M$ be a  finitely generated $R$-module.  We show $M$ is holonomic if and only if $M$  is free as an  $R$-module. \end{observation}
So, Observation \ref{1.4} presents another connection of  commutative 
algebra to 
D-Modules. Namely, see the standard text books \cite[3.3.2]{dd} and \cite[Theorem 1.1.25]{b} where Observation \ref{1.4} proved by applying well-known existence
of solutions to a certain differential system. This motivates us to present the prime characteristic analogue of 
Observation \ref{1.4} in Corollary \ref{PRIME}. It seems this is new, as \cite{b,dd} are about of zero characteristic rings.

According to group theory, we know  the size of homological objects $\Hom_\mathbb{Z}(-,\mathbb{Z})$ and $\Ext^1_\mathbb{Z}(-,\mathbb{Z})$ may force
something  about the size of the abelian group $(-)$. The initial works in this area are due to Nunke \cite{n} and Jensen \cite{j}. Following these, recall that
a module $\mathcal{M}$ is called cohomologically finite (length) if
$\Ext^i_R(\mathcal{M},R)$ is finite (length) for all $i$. \S  5 deals with cohomologically finite (length) modules, and connect  them to finite (length) modules. For instance, see   Proposition \ref{br}.
 This is inspired from a result of Bredon from his famous book on sheaf theory \cite{br}, and suggests the following problem:  
\begin{question} \label{1.5}Suppose we know 
some properties of	$\Ext^i_R(-,R)$ 	for all $i$. What can say about the corresponding property of the module $(-)$?\end{question}

Concerning Question \ref{1.5}, some properties of modules descent from the $\Ext$-family to $(-)$  if
the topology of $R$ is reach, e.g., it is complete with respect to adic-topology.
A module $\mathcal{M}$ is called cohomologically zero if
$\Ext^i_R(\mathcal{M},R)=0$  for all $i$. For  a finitely generated module,  zero and cohomologically zero are equivalent  notations.
The finitely generated assumption of $\mathcal{M}$ is important.
For more information on this  see \cite{a}.  Despite this,  we present more evidences on Question \ref{1.5}.

In \S 6  
 we deal with cohomologically zero injective modules.
 Again,  this is  motivated from Auslander's comments on the functor Ext, 
where among other things, he proved the cohomologically zero property of $\mathcal{Q}(R)$, where  $R$ is  a complete local domain.
We extend this result and present some partial converses. Then, we   study  the cohomological zero property of  $E_R(R/\fp)$
where $\fp\in\Spec(R).$  To state an application, recall that Osofsky \cite{os} computed projective dimension of injective envelop of $R$ when  $R$ is a local ring of a polynomial ring over $\mathbb{R}$ in $m$ variables with $m\geq n+3$.  Her answer is very interesting: $$\pd_R(\mathcal{Q}(R))=n+2 \Longleftrightarrow2^{\aleph_n}=\aleph_{n+1}.$$
So, it depends to  accept generalized continuum hypothesis or not.  What is $\pd_R(\mathcal{Q}(R))$?  A domain is called Matlis provided
$\pd_R(\mathcal{Q}(R))=1$. Matlis  results \cite{matcanada} show that we may assume that the Krull dimension of $R$ is bigger than $1$. Also, we may assume the rings are uncountable, because any countable domain is Matlis, see \cite{kapc}. The case of $2$-dimensional regular rings may follow by   combining the work of Kaplansky \cite{kapc} along with Noether's normalization theorem (also, see Example \ref{2com}). We encourage the interested reader  to compute the first syzygy module of the fraction field of  $R:	= k[ x,y,z ]_{\fm}/(x^2 + y^ 3 + z^5 )$. 

\begin{example}	\label{1.6}
Let $R:	= \mathbb{C}[ x,y,z ]_{\fm}/(x^2 + y^ 3 + z^5 )$. Then $\pd_R(\mathcal{Q}(R))=2$.
\end{example}

  More interestingly, a natural question arises:

\begin{question}\label{qid}
	 What is
 $\pd_R(E_R(R/ \fp))$?
\end{question}
 We close \S 6 by answering this in some cases.
 It was well-known to Matlis that $$\pd_R(E_R(k))\leq \dim R\quad(\ast),$$ when $R$ is Gorenstein and local.
 One may restate this as:

\begin{fact}
	Let $(R,\fm)$ be a regular  ring  and $\fp$ be a prime ideal. Then  $\pd_R(E_R(R/ \fp))\geq \Ht(\fp)$.
\end{fact}
Here, we present some nontrivial situations for which the bound $\pd_R(E_R(R/ \fp))\geq \Ht(\fp)$ is (not) sharp. 
	For instance, suppose $R$ is regular and of dimension $d$,  we show

	\begin{enumerate}
	\item[a)]   $\pd_R(E_R(R/ \fp))=d$ when $\fp$ is of height $d-1$ and $R/ \fp$ is not complete.
\end{enumerate}

We start \S7  by a connection to the realization problem of Nunke \cite{n}. This may consider as an application of \S6. Our elementary approach reproves some technical (and important) results of Jensen \cite{j2}. Additional properties of (homologically) realizable  modules are given.
We do these tasks by applying some basic properties of local cohomology modules, a tool introduced by Grothendieck some years after than \cite{n}:

\begin{observation} \label{18}
	Let  $(R,\fm)$  be a  complete local ring,  $\underline{x}:=x_1,\ldots,x_n$ be a regular sequence
	and let $M$ be finitely generated. Then $M\cong \Ext^i_R(\HH^i_{{\underline{x}_i}}(R),M)$  for all $0<i\leq n$. In particular,
	$M$ is  realizable at level $i$. 
\end{observation}

According to the argument of Observation \ref{18},
Matlis duality has an application to the realization problem. If   we apply it for the very special case that the ring is
1-dimensional Gorenstein local integral domain, then we recover a result of Matlis, see \cite[Page 580, Cor]{matcanada}. 
We now know that the argument behind in the Observation \ref{18} is not new. However, it may simply 
some important results of Matlis \cite{matne}.
Here, we present an application of this simplification.
Indeed,  it fits into 
Question \ref{qid}, and determines the situations for which the bound $(\ast)$ is sharp:

	\begin{enumerate}
	\item[b)] $\pd_R(E_R(k))= \dim R$ when $R$ is Gorenstein,
		\item[c)]  $\pd_R(E_R(k))< \infty$ iff $R$ is Gorenstein.
\end{enumerate}Observation b) was  proved originally by Matlis as an applying his  theory developed in  \cite{matne}.
Observation c) was proved by Matlis \cite[Theorem 6]{matcanada} when $R$ is a one-dimensional integral domain. 
Also, he  extended it to the Cohen-Macaulay case, see  \cite[Corollary 3.13]{matne}.
In particular, our simplification drops this Cohen-Macaulay assumption.
This allows us   to show:

\begin{corollary}
	Let $(R,\fm)$ be excellent
	and let $I$ be the defining ideal of non Cohen--Macaulay locus of   $R$.
	 Then $\pd_R(E_R(R/\fp))= \infty$
	for all $\fp\in\V(I)$.
\end{corollary}

In \S8  
we study cohomologically zero flat modules.
Concerning this, there are series of interesting works on a result of Jensen  \cite{j2} (reproved by  Buchweitz and Flenner \cite[Corollary 1]{bf}). Both proofs are not trivial.
In Corollary \ref{f}, we use an easy natural transformation originally due  to  Auslander,
to simplify \cite[Corollary 1]{bf}. We apply the realization problem to present 
another homological property of Dedekind domains. This was started by the seminal works
  \cite{n} and  \cite{j2}. Recall that this section has an application to Question \ref{qid}.

In Section 9,  we give a  splitting criteria for   short exact sequences of torsion-free modules. Again, some local cohomology arguments appear. \begin{proposition}\label{splitmat1}
	Let $(R,\fm)$ be a local domain and $f\in R$ be nonzero.
	Let $\zeta:=0\to R\to \mathcal{A} \to R_f\to 0$ be such that $\mathcal{A}$  decomposes into nonzero modules.
	Then $\zeta$ splits.
\end{proposition}
Matlis introduced the concept of $D$-rings, and he asked are $D$-rings complete 
in $R$-topology.  Kaplansky \cite{k} answered this affirmatively. As an application of Proposition \ref{splitmat1}, we present a new proof of this, namely   $D$-rings are complete 
in $R$-topology.
We close Section 9 by simplifying an splitting theorem of Chase \cite[Theorem 3.4]{Chase} and extending it removing the  integral domain domain assumption.

There are some weaker versions of homologically zero modules.
For example, the vanishing property $\Ext^i_R(\mathcal{M},R)=0$ is valid only for a restricted  range of $i$.
In Section 10 we deal with the following question: 

\begin{question}\label{101}(Gerstner)
	For which property of $R$ to
	ensure, that condition
	$\Hom_R(M,R)=\Ext^1_R(M,R)=0$
	on $R$-modules $M$ implies $M =0$. 
\end{question}

In Section 11 we comeback to  $\pd_R(\mathcal{Q}(R))$. Let us revisit Example \ref{1.6}, and apply its method to answer the question. Namely, we prove:

\begin{theorem}\label{1.13}
		Let $R$ be an affine integral domain of
	dimension $d$ over $\mathbb{C}$. The following assertions are valid:
	\begin{enumerate}	
		\item[a)] Suppose $d=1$. Then $\pd_{R}(\mathcal{Q}(R))=1$.
		\item[b)] If $d=2$ and $R$ is $UFD$, then $\pd_{R}(\mathcal{Q}(R))=2$.
		\item[c)] Adopt $\textbf{CH}$. Then $\pd_{R}(\mathcal{Q}(R))\leq2$.
		\item[d)]	Assume in addition to c) that  $R$ is $UFD$ and of dimension bigger than one. Then $\pd_{R}(\mathcal{Q}(R))=2$.
	\end{enumerate}
\end{theorem}

In Theorem \ref{113}, we present a local version of Theorem \ref{1.13}. It may be nice to note that this  local version is in the prestige work \cite[ \S 3.3]{GR} by 	Raynaud and  Gruson where they computed  projective dimension of some flat modules. 
 
 In the final section, we extend the following result of Matlts in two different directions. Namely, the rank consideration and the Krull-dimension restriction:
 
 \begin{corollary}(Matlis)
 	Let $(R,\fm)$ be a 1-dimensional complete domain with  fraction filed $Q$. Let $B\subset Q$ and $B'\subset Q$ both congaing $R$. If $Q/B\cong Q/ B'$ then $B\cong B'$.
 \end{corollary}
 
 So, one may recover a proposed isomorphism
 after applying cohomological functors, like local cohomology and certain tor-modules.

Despite these, I hope our elementary   approach  sheds more light  on the topics quoted  here.

\section{Roots of cohomological splitting}

Let $(R,\fm,k)$ be a commutative noetherian local ring.
Let $\zeta:=0\to \mathcal{M}_1\to\mathcal{M} \to \mathcal{M}_2\to 0$ be a short exact sequence of  modules and
let $\mathcal{L}$ be a module. By $\mathcal{L}\otimes_R\zeta$ we mean the complex $$0\lo \mathcal{L}\otimes_R\mathcal{M}_1\lo \mathcal{L}\otimes_R\mathcal{M} \lo \mathcal{L}\otimes_R\mathcal{M}_2\lo 0$$
which is not necessarily exact. Similarly, $\Hom_R(\mathcal{L},\zeta)$  stands for the complex $$0\to \Hom_R(\mathcal{L},\mathcal{M}_1)\to \Hom_R(\mathcal{L},\mathcal{M}) \to \Hom_R(\mathcal{L},\mathcal{M}_2)\to 0.$$

\begin{fact} (See \cite[Theorem 3.28]{v}) \label{csp}
	If $M\cong M_1\oplus M_2$ and all modules under consideration are finitely generated, then $\zeta$ splits.
\end{fact}
In the case of maximal ideals the next result was proved by Striuli in her thesis \cite{st2}.
Here, we extend it by a simple argument:

\begin{observation}\label{o1} Let $(R,\fm)$ be local. Let $\fa\lhd  R$ and $\zeta:=0\to M_1\to M \to M_2\to 0$ be a short exact sequence of finitely generated modules. Then
$\zeta$ splits if and only if $\frac{R}{\fa^n}\otimes_R\zeta$ splits for all $n\gg 0$.
\end{observation}

\begin{proof}Let $\fa\lhd  R$.
 Suppose $\frac{R}{\fa^n}\otimes_R\zeta$ splits  for all $n\gg 0$. 
 By $\widehat{-}$ we mean $\fm$-adic completion. Thanks to Fact \ref{csp}
 $\zeta$ splits if we can show that $M\cong M_1\oplus M_2$.
 This is the case provided $\widehat{M}\cong \widehat{M_1}\oplus\widehat{M_2}$,
since
 modules are finitely generated. Then we may assume that $R$ is $\fm$-adically complete.
 In particular, $R$ is $\fa$-adically complete.  
 Since $\frac{R}{\fa^n}\otimes_R\zeta$ splits for all $n\gg 0$, we have $$\frac{M}{\fa^n M}\cong \frac{M_1}{\fa^n M_1}\oplus \frac{M_2}{\fa^n M_2}$$
  for all $n\gg 0$. Recall that completion commutes with finite direct sums. Taking inverse limit, $M^{\widehat\fa}\cong M^{\widehat\fa}_1\oplus M^{\widehat\fa}_2$. Since
  modules are finitely generated, $M\cong M_1\oplus M_2$. This completes the proof.
\end{proof}
\begin{notation}Let $(R,\fm,k)$ be local. The notation $(-)^v$ stands for the Matlis functor $\Hom_R(-,E_R(k))$, 
where  $ E_R(k)$ denotes the injective envelope of $k$.\end{notation}

\begin{observation} \label{o1d} Let $(R,\fm)$ be local.
Let $\fa\lhd  R$ and 
 $\xi:=0\to \mathcal{M}_1\to \mathcal{M} \to\mathcal{M}_2\to 0$ be a short exact sequence of artinian modules.
Then $\xi$ splits if and only if $\Hom_R(\frac{R}{\fa^n},\xi)$ splits for all $n\gg 0$.
\end{observation}

\begin{proof}
	Suppose $\Hom_R(\frac{R}{\fa^n},\xi)$ splits  for all $n\gg 0$. It is easy to 
	any artinian module has the structure of $\widehat{R}$-module, compatible with the $R$-module
	structure. From this, and without loss of the generality, we may
	assume that $R$ is complete. 
Recall that $\Hom_R(\frac{R}{\fa^n},\xi)^v$ splits.
	In view of Hom-Evaluation,  $\frac{R}{\fa^n}\otimes_R\xi^v$ splits for all $n\gg 0$.
According to Matlis theory, 	$\zeta:=\xi^v$ is  a short exact sequence of finitely generated modules.      Due to Observation  \ref{o1}, $\zeta$ splits. Taking another
	Matlis duality, we see that 
	$\xi=\zeta^{vv}$ splits.
\end{proof}
Concerning Observation \ref{o1d} (resp. \ref{o1}), the    artinian    (resp. finitely generated) assumption of modules in $\xi$ (resp. $\zeta$) is important.
To this end,  by $\Syz_n(-)$ we mean the $n$-th syzygy module of $(-)$. 

\begin{example}
Let $(R,\fm)$ be a local integral domain of dimension $d\geq 2$. We look at the exact sequence $$\xi:=0\lo\Syz_1(\fm)\lo R^{\beta_0(\fm)}\lo \fm\lo 0$$ of finitely generated modules.
It is easy to see that  $\Hom_R(\frac{R}{\fm^n},\xi)$ splits, in fact the complex 
$\Hom_R(\frac{R}{\fm^n},\xi)$ consists of zero modules, because $\xi$ consists of  modules of positive depth. If  $\xi$ splits, it follows by definition that $\fm$ is projective. It turns out that
$\fm$ is principal, a contradiction with the  generalized Krull's principal ideal theorem.
\end{example}

\begin{definition}  A short exact sequence $\zeta:=0\to \mathcal{M}_1\to \mathcal{M} \to \mathcal{M}_2\to 0$  is called   E-splits at level $i$ with respect to $\fa$, if
$$0\lo \Ext^i_R(R/ \fa^n, \mathcal{M}_1) \lo \Ext^i_R(R/ \fa^n, \mathcal{M}) \lo \Ext^i_R(R/ \fa^n, \mathcal{M}_2)	\lo 0$$
splits for all  $n\gg 0$.
Similarly, using $\Tor$-functor instead of $\Ext$-functor,  one may define  T-splitting with respect to $\fa$. 
\end{definition}

\begin{lemma}\label{dmat} Let $(R,\fm)$ be local.
Suppose a short exact sequence $\zeta:=0\to \mathcal{M}_1\to \mathcal{M} \to \mathcal{M}_2\to 0$ is   $E$-split  at level $i$ with respect to $\fa$. Then
the dual sequence	$\zeta^v$  $T$-splits  at level $i$  with respect to $\fa$. The converse holds if  $\zeta$
	consists of finitely generated modules.
\end{lemma}  

\begin{proof}
By definition, $0\to \Ext^i_R(R/ \fa^n, \mathcal{M}_1) \to \Ext^i_R(R/ \fa^n, \mathcal{M}) \to \Ext^i_R(R/ \fa^n, \mathcal{M}_2)	\to 0$ splits. In  particular its Matlis dual splits. Recall  that
 $\Ext^i_R(R/ \fa^n,\mathcal{L})^v\cong \Tor_i^R(R/ \fa^n,\mathcal{L}^v)$. From this,  $\zeta^v$  is of T-splitting type at level $i$.
 For  the converse,
without loss of the generality we  assume that $R$ is complete, and note that 	$\zeta^{vv}=\zeta$.
\end{proof}

\begin{remark}
One may extend the converse part of Lemma \ref{dmat} by assuming $\zeta$
consists of Matlis reflexive modules. Recall  by a result of Enochs and
 $Zoschinger$ that $\mathcal{M}$ is Matlis reflexive iff $\mathcal{M}$
contains a finitely generated submodule $N$ such that $\mathcal{M}/N$ is artinian.
\end{remark}

Here, is  a sample of  E-splitting (resp. T-splitting):

\begin{proposition}\label{esplit} Let $(R,\fm)$ be local.
	Let $\zeta:=0\to M_1\to M \to M_2\to 0$ be the exact sequence  of finitely generated modules and let $\fa$ be an ideal generated by a regular element $x$. If $\zeta$  E-splits at level one with respect to $\fa$,
	then $\zeta$ splits.
\end{proposition}

\begin{proof}
	The free resolution of $R/\fa ^n$ is given
	by $0\to R\stackrel{x^n}\lo R\to R /\fa^n \to 0$. Then for  any $\mathcal{L}$, 
we have  $$\Ext^1_R(R/ \fa^n, \mathcal{L}) =\HH(\mathcal{L}\stackrel{x^n}\lo\mathcal{L}\to 0 )=\mathcal{L}/ x^n \mathcal{L}.$$
The assumption says that the sequence $$0\lo \Ext^1_R(R/ \fa^n, M_1) \lo \Ext^1_R(R/ \fa^n, M) \lo \Ext^1_R(R/ \fa^n, M_2)	\lo 0$$splits, and in particular is exact. Combining these, we observe that the sequence
$$0\lo M_1/ x^n M_1 \lo M / x^nM\lo M_2/ x^nM_1\lo 0$$  splits, and so exact,
i.e., $\frac{R}{\fa^n}\otimes_R\zeta$ splits for all $n\gg 0$. In view of 
Observation \ref{o1} 
$\zeta$ splits.
\end{proof}

\begin{proposition}\label{tsplit} Let $(R,\fm)$ be local.
	Let $\zeta:=0\to \mathcal{M}_1\to \mathcal{M} \to \mathcal{M}_2\to 0$ be the exact sequence  of artinian modules and let $\fa$ be an ideal generated by a regular element $x$. If $\zeta$  T-splits at level one with respect to $\fa$,
	then $\zeta$ splits.
\end{proposition}

\begin{proof}
	One may use  Proposition \ref{esplit} and Lemma \ref{dmat}.
	Here, is a direct proof. The free resolution of $R/\fa ^n$ is given
	by $0\to R\stackrel{x^n}\lo R\to R /\fa^n \to 0$. Then for  any $\mathcal{L}$, 
	we have  $$\Tor_1^R(R/ \fa^n, \mathcal{L}) =H((0\to R\stackrel{x^n}\lo R)\otimes_R\mathcal{L})=\ker(\mathcal{L}\stackrel{x^n}\lo \mathcal{L})=(0:_\mathcal{L}x^n )\cong\Hom_R(R/x^nR, \mathcal{L})\quad(\ast)$$
	The assumption says that $$0\lo \Tor_1^R(R/ \fa^n, \mathcal{M}_1) \lo \Tor_1^R(R/ \fa^n, \mathcal{M}) \lo \Tor_1^R(R/ \fa^n, \mathcal{M}_2)	 \lo 0$$ is split exact. Combine this with $(\ast)$, we observe that 
 $\Hom_R(\frac{R}{\fa^n},\zeta)$ splits for all $n\gg 0$. According to
	Observation \ref{o1d} we know that
	$\zeta$ splits.
\end{proof}

\begin{definition}
Let $\mz\subset\Spec R$ and   $\fp \in \mz$.
Then $\mz$ is called specialization-closed if $\V(\fp)\subset \mz$.
\end{definition}

\begin{example}\label{tort}	
	Here, are some examples of specialization-closed sets:
	\begin{enumerate}
		\item[a)] Every closed subset of $\Spec(R)$ with respect to Zariski topology is specialization-closed.
		\item[b)] The set $\Spec(R)\setminus\min(R)$ is specialization-closed.
		\item[c)] Suppose $R$ is equidimensional. The set $\Spec(R)\setminus\Ass(R)$ is specialization-closed.
	\end{enumerate}
\end{example}

Let $\mz\subset \Spec R$ be  specialization-closed, and  define $$\Gamma_{\mz}(M):=\{m\in M\mid\Supp_R(Rm)\subseteq\mz\}.$$
For $i\in\NN_0$, the $i$-th right derived functor of $\Gamma_{\mz}(-)$, denoted
by $\HH^i_{\mz}(-)$. 
In the case $\mz:=\V(\mathfrak {\fa})$ this is $\HH^i_{\mathfrak{a}}(-)$.
Here, we state two easy facts about them:

\begin{fact}
	Let $(R,\fm)$ be equidimensional, and let $\mz:=\Spec(R)\setminus\Ass(R)$. Then  $\Gamma_{\mz}(L)=\tor(L)$.
\end{fact}

\begin{fact}\label{th}
	Let $(R,\fm)$ be $1$-dimensional and equidimensional. Then $\HH^0_{\fm}(L)=\Gamma_{\mz}(L)=\tor(L)$.
\end{fact} 

\begin{definition}
	Let $\zeta:=0\to M_1\to M \to M_2\to 0$ be a short exact sequence of (finitely generated) modules. Then	$\zeta$ is called cohomologically splits with respect to $\fa$ provided
	$$0\lo \HH^i_{\fa}(M_1) \lo \HH^i_{\fa}(M) \lo \HH^i_{\fa}( M_2)	\lo 0$$
	splits for all $i$.
\end{definition}

\begin{observation} 
	$E$-splitting implies  cohomologically splitting.
\end{observation}

\begin{proof}
This follows by $\HH^i_{\fa}(-)=\vil_n\Ext^i_R(R/ \fa^n, -)$.
\end{proof}

Here, is  a sample to checking the  cohomologically splitting property.

\begin{observation}\label{STAR}
	Let $(R,\fm)$ be a $3$-dimensional quasi-Gorenstein local ring and let $$\zeta:= 0\lo M_1\lo M \lo M_2\lo 0$$ be an exact sequence  of finitely generated  modules such that $\HH^{d-1}_{\fm}(M)=\HH^{d-1}_{\fm}(M_1)\oplus\HH^{d-1}_{\fm}(M_2)$.  If  $M_1$ is free, then $\zeta$ cohomologically splits with respect to $\fm$.
\end{observation}

\begin{proof}
Quasi-Gorenstein rings satisfy Serre's  condition $(S_2)$. From this, $\HH^0_{\fm}(M_1)=\HH^1_{\fm}(M_1)=0$. In particular, $$0\to 0=\HH^0_{\fm}(M_1) \to \HH^0_{\fm}(M) \to \HH^0_{\fm}( M_2)	\to \HH^1_{\fm}(M_1)=0$$splits. 
	We look at the long exact sequence $$0=\HH^1_{\fm}(M_1)\to \HH^1_{\fm}(M)\to \HH^1_{\fm}(M_2)\to \HH^2_{\fm}(M_1)\stackrel{f}\lo \HH^2_{\fm}(M)\to \HH^2_{\fm}(M_2)$$
	to deduce that $\HH^1_{\fm}(M)\to \HH^1_{\fm}(M_2)$ is surjective, because $f$ is injective.
By this, $$0= \HH^1_{\fm}(M_1) \lo \HH^1_{\fm}(M) \lo \HH^1_{\fm}( M_2)	\lo 0$$ splits. The splitting in the second spot is given by the assumption.
	We look at $$\HH^2_{\fm}(M)\stackrel{g}\to \HH^2_{\fm}(M_2)\to \HH^3_{\fm}(M_1)\to \HH^3_{\fm}(M)\to \HH^3_{\fm}(M_2)\to 0.$$
	Since $g$ is surjective,  we know$$\xi:=0\to \HH^3_{\fm}(M_1)\to \HH^3_{\fm}(M)\to \HH^3_{\fm}(M_2)\to 0$$ is exact. Since
	$M_1$ is free, $\HH^3_{\fm}(M_1)$ is injective as an $R$-module. Thus $\xi$ splits. In sum, $\zeta$ cohomologically splits with respect to $\fm$.
\end{proof}

In Corollary \ref{chigh} (resp. \ref{quasi}) we remove the reflexivity (resp. Cohen-Macaulay) assumption of:

\begin{corollary}\label{213}
	Let $(R,\fm)$ be a $d$-dimensional Gorenstein local ring and let  $\zeta:= 0\to M_1\to M \to M_2\to 0$ be an exact sequence   of finitely generated reflexive modules such that $\HH^{d-1}_{\fm}(M)=\HH^{d-1}_{\fm}(M_1)\oplus\HH^{d-1}_{\fm}(M_2)$. If  $M_1$ is free, then $\zeta$ splits.
\end{corollary}
The notation $(-)^\ast$ stands for $\Hom_R(-,R)$.

\begin{proof}
	According to  the proof of Observation \ref{STAR} the sequence 
	$$0\lo \HH^d_{\fm}(M_1)\lo \HH^d_{\fm}(M)\lo \HH^d_{\fm}(M_2)\lo 0$$ splits. Also, its Matlis dual  splits. 
	In the light  of local duality we observe that $M^\ast\cong M_1^\ast\oplus M_2^\ast$.
	Taking another $(-)^\ast$ yields that $M^{\ast\ast}\cong M_1^{\ast\ast}\oplus M_2^{\ast\ast}$.
	By reflexivity,  $$M\cong M^{\ast\ast}\cong M_1^{\ast\ast}\oplus M_2^{\ast\ast}\cong M_1\oplus M_2.$$ According to Fact \ref{csp}
	$\zeta$ splits.
\end{proof}

By $\pd_R(-)$   we mean the projective dimension  of $(-)$.
Corollary \ref{213} is not true if instead of $M_1$ we assume $M$ is free:

\begin{example}\label{dgeq1}
	Let $(R,\fm,k)$ be a Cohen-Macaulay local ring of dimension $d>1$ which is not regular, and look at $\zeta:=0\to\Syz_{d+1}(k)\to R^{\beta_d(k)}\to \Syz_{d}(k)\to 0$. 
Then 	\begin{enumerate}
	\item[a)] $\zeta$ consists of
	reflexive modules.
	\item[b)]  $\HH^{d-1}_{\fm}(M)=\HH^{d-1}_{\fm}(M_1)\oplus\HH^{d-1}_{\fm}(M_2)$.
	\item[c)] $\zeta$ does not split.
\end{enumerate}	 
\end{example}

\begin{proof}Recall that $\zeta$ is a sequence of maximal
	Cohen-Macaulay modules, and so they are reflexive modules, because they are second syzygy modules as $d>1$.  In view of definitions we observe that $$0=\HH^{d-1}_{\fm}(M)\cong\HH^{d-1}_{\fm}(M_1)\oplus\HH^{d-1}_{\fm}(M_2)=0  .$$ Now, suppose on the way of contradiction that $\zeta$ splits. By definition, 
	$\Syz_{d+1}(k)$ is projective, e.g., $\pd_R(k)<\infty$, and consequently, $R$ is regular. This is excluded by the assumption.
\end{proof}

\section{Cohomological splitting}

We start with

\begin{question} \label{1.2}(See \cite[Page 78]{v})
	Let $R$ be $1$-dimensional Cohen-Macaulay local  ring and let  $\zeta$ be the exact sequence $ 0\to M_1\to M \to M_2\to 0$   of finitely generated modules of projective dimension at most one such that $$\tor(M)=\tor(M_1)\oplus\tor(M_2).$$ Is $\zeta$ split?
	\end{question}

\begin{proposition}\label{o3}
	Question \ref{1.2}
is true if   $M$ is torsion-free.
\end{proposition}

\begin{proof}
 In view of $\tor(M)=\tor(M_1)\oplus\tor(M_2)$
	we deduce that $\tor(M_2)=0$. Without loss of the generality we assume that $M_2\neq 0$.
	This implies that $M_2$ is of positive depth. Thanks to Auslander-Buchsbaum formula,   we know $M_2$ is  projective. By definition, the sequence   $0\to M_1\to M\to M_2\to 0$
splits, as  claimed.
\end{proof}

The above result shows the assumption $d>1$ is essential in Example \ref{dgeq1}. Its proof shows:

\begin{observation}
	Question \ref{1.2}
	is true if in addition $M_2$ is torsion-free.
\end{observation}

By  $\id_R(-)$  we mean the  injective dimension  of $(-)$.

\begin{proposition}\label{o3.c}
	Let $R$ be a local ring of dimension one, $M$ be torsion-free and $M_1$ be of finite injective dimension.
	Let  $$\zeta:= 0\lo M_1\lo M \lo M_2\lo 0$$ be an exact sequence   of finitely generated  modules. If $\tor(M)=\tor(M_1)\oplus\tor(M_2)$, then $\zeta$ splits.
\end{proposition}

\begin{proof}According to a theorem of Bass, $R$ is Cohen-Macaulay. Also, $\id_R(M_1)=\depth_R(R)=1$.
	From $$\tor(M)=\tor(M_1)\oplus\tor(M_2)$$ we deduce that
	$\tor(M_i)=0$. 
	In particular, if $y$ is a regular element of $R$, it is regular over any of $\{M_1,M_2\}$. Such a thing exists, as $R$ is Cohen-Macaulay.
We repeat this for any powers of $y^n$, and we set $\overline{R}:=\frac{R}{y^n R}$.  It is easy to see $\Tor_1^R( \overline{R},M_2)=\ker (M_2\stackrel{y^n}\lo M_2)=0.$ 
	We apply $-\otimes_R  \overline{R}$ to
	$\zeta$ and obtain $$0=\Tor_1^R(R/y^nR,M_2)\lo M_1\otimes_R  \overline{R}\lo M\otimes_R \overline{R}\lo M_2\otimes_R  \overline{R}\lo 0.$$
	Since $\id_{\overline{R}}(M_2 \otimes_R  \overline{R})= \id_{R}(M_2 )-1=0$, the above sequence splits as an $\overline{R}$-module.
	Recall that $$\Hom_R( M\otimes_R \overline{R}, M_1\otimes_R \overline{R})=\Hom_{\overline{R}}( M\otimes_R \overline{R}, M_1\otimes_R \overline{R}),$$ e.g.,
	the above sequence splits as an $R$-module.  Now, recall  from \cite[Proposition 2.8]{st} that:
	\begin{enumerate}
		\item[Fact]A):  Let $A, B$  be finitely generated, let $x\in R$ be a non-zerodivisor
		on $R$, $A, B$ and let $\alpha:=0\to B\to X_{\alpha}\to A\to 0$ be the short exact sequence.
		Suppose   $\alpha\otimes R/xR$ splits. Then $\alpha\in x \Ext^1_R(A,B)$.
	\end{enumerate}
	In the light of Fact A) we see   $$\zeta\in \bigcap _{n\in\mathbb{N}}y^n \Ext^1_R(M_2,M_1)=0,$$  because of the Krull's intersection theorem. So, $\zeta$ splits.
\end{proof}

\begin{example}
	Let $(R,\fm,k)$ be a 1-dimensional complete local domain which is not Gorenstein. There is an exact sequence$$0\lo M_1\lo M\lo M_2\lo 0$$ of finitely generated modules such that $M_2$ is of finite injective dimension and $$\tor(M)=\tor(M_1)\oplus\tor(M_2),$$ but $M\neq M_1\oplus M_2.$
\end{example}

\begin{proof}
Let $\omega$ be the canonical module.
Since it is not free, the first betti number is not zero.
So,$$\Ext^1_R(\omega, k)^v=\Ext^1_R(\omega, k^v)^v=\Tor^R_1(\omega, k)^{vv}=\Tor^R_1(\omega, k)\neq 0.$$	
Let $\zeta$ be any nonzero element of
$\Ext_R^1(\omega,k)$.
By Baer-Yoneda, there is an $R$-module $K$ such that
	$$\zeta:=0\lo k\lo K\lo \omega\lo 0,$$
	e.g., it does not split. By applying the long exact sequence of local cohomology
modules	we have: $$0\lo\HH^0_{\fm}(k) \lo \HH^0_{\fm}(K) \lo \HH^0_{\fm}(\omega)	=0 .$$In other words,
	 $\tor(K)=k=\tor(k)\oplus\tor(\omega)$.
	 But, $M\neq M_1\oplus M_2,$
	 because $\zeta$ is nonzero.
\end{proof}

\begin{example}
	Let $(R,\fm,k)$ be a 1-dimensional  local domain which is not regular. There is an exact sequence $0\to M_1\to M\to M_2\to 0$ of finitely generated modules such that   $\tor(M)=\tor(M_1)\oplus\tor(M_2),$ but $M\neq M_1\oplus M_2.$
\end{example}

\begin{proof}
	Let $\Omega:=\Syz_1(k)$ and apply the previous proof for $\Omega$ instead of $\omega	$.
\end{proof}
\begin{notation}
	By $\mathcal{Q}(R)$ we mean  the fraction field  of a local domain $R$.
\end{notation}
In Question \ref{1.2} the finitely generated assumption is needed:

\begin{example}
	Let $(R,\fm)$ be a 1-dimensional  local domain which is not complete. There is an exact sequence$$0\lo \mathcal{M}_1\lo \mathcal{M}\lo \mathcal{M}_2\lo 0$$ of  modules  of finite projective dimension such that $$\tor(\mathcal{M})=\tor(\mathcal{M}_1)\oplus\tor(\mathcal{M}_2),$$ but $\mathcal{M}\neq \mathcal{M}_1\oplus \mathcal{M}_2.$
\end{example}

\begin{proof}
	Since $\dim(R)=1$, and
	by a result of Matlis, $\pd(\mathcal{Q}(R))=1$.
	Due to non-complete assumption, and by applying another result of Matlis,
	we know, $\Ext^1_R(\mathcal{Q}(R),R)\neq 0$. By Baer-Yoneda,
	 there is an $R$-module $\mathcal{M}$ such that
	$$\zeta:=0\lo R\lo \mathcal{M}\lo \mathcal{Q}(R)\lo 0,$$
	is exact and not splits. It is easy to see that $\mathcal{M}$ is of finite projective dimension. By applying the long exact sequence of local cohomology
	modules	we have: $$0=\HH^0_{\fm}(R) \lo \HH^0_{\fm}(\mathcal{M}) \lo \HH^0_{\fm}(\mathcal{Q}(R))	=0 .$$In other words,
	$\tor(\mathcal{M})=0=\tor(R)\oplus\tor(\mathcal{Q}(R))$.
	But, $\mathcal{M}\neq R\oplus \mathcal{Q}(R),$
	because $\zeta$ is nonzero.
\end{proof}

In order to present a higher dimensional version of Theorem \ref{1.1}, we borrow some lines from \cite{v}.

\begin{theorem}\label{high}
	Let $(R,\fm)$ be a $d$-dimensional Gorenstein local ring  with $d>0$ and let $\zeta$ be the exact sequence $ 0\to M_1\to M \to M_2\to 0$   of finitely generated modules of projective dimension at most $1$. If $\HH^{d-1}_{\fm}(M)\cong\HH^{d-1}_{\fm}(M_1)\oplus\HH^{d-1}_{\fm}(M_2)$,
	then $\zeta$ splits.
\end{theorem}

\begin{proof}
	We apply the local duality theorem to see  $\Ext^1_R(M,R)\cong\oplus_i\Ext^1_R(M_i,R)$.
	This yields $$\Ext^1_R(M,R)\otimes_RM_1\cong\oplus_i\Ext^1_R(M_i,R)\otimes_RM_1.$$
Let $\mathcal{C}$
be the family of finitely generated  modules. Let $F:\mathcal{C}\to Ab$ be the functor
defined by 
 $F(L):=\Ext^1_R(M_1,L)$. Since $\pd(M_1)$ is at most $1$,
	this functor is right exact. Also, it  preserves direct sums. By Watts' theorem, $$\Ext^1_R(M_1,L)=F(L)\cong F(R)\otimes_RL=\Ext^1_R(M_1,R)\otimes_RL.$$
	We apply this
	to deduce that  
	$ \Ext^1_R(M_2,M_1)\oplus \Ext^1_R(M_1,M_1)\cong\Ext^1_R(M,M_1)$.
	Thanks to Fact \ref{csp}  the following natural  sequence 
	$$0\lo \Ext^1_R(M_2,M_1)\lo \Ext^1_R(M,M_1)\lo \Ext^1_R(M_1,M_1)\lo 0,$$
	is exact. We apply this along with the long exact sequence of $\Ext$-modules to conclude   $\Hom_R(M,M_1)\to \Hom_R(M_1,M_1) $  is surjective.
	By looking at the preimage of $1\in\Hom_R(M_1,M_1)$,    $\zeta$ splits.
\end{proof}

\begin{corollary}\label{chigh}
	Let $(R,\fm)$ be a $d$-dimensional Gorenstein local ring and let  $\zeta:= 0\to M_1\to M \to M_2\to 0$ be an exact sequence   of finitely generated  modules such that $\HH^{d-1}_{\fm}(M)\cong\HH^{d-1}_{\fm}(M_1)\oplus\HH^{d-1}_{\fm}(M_2)$. If  $M_1$ is free, then $\zeta$ splits.
\end{corollary}

\begin{proof}
	This is in the proof of  Theorem \ref{high}.
\end{proof}

\begin{corollary}\label{quasi}
	Let $(R,\fm,k)$ be a $d$-dimensional quasi-Gorenstein local ring and let  $$\zeta:= 0\lo M_1\lo M \lo M_2\lo 0$$ be an exact sequence   of finitely generated reflexive modules such that $\HH^{d-1}_{\fm}(M)=\HH^{d-1}_{\fm}(M_1)\oplus\HH^{d-1}_{\fm}(M_2)$. If  $M_1$ is free, then $\zeta$ splits.
\end{corollary}

\begin{proof} 
	A ring is quasi-Gorenstein if and only if its completion is quasi-Gorenstein.
	Let $L_1:=M$  and $L_2:=M_1\oplus M_2$ be finitely generated modules.   If $L_1\otimes_R\widehat{R}\cong L_2\otimes_R\widehat{R}$, then $L_1\cong L_2$.
According to Fact \ref{csp}, we may and do assume that $R$ is complete.
It follows from the proof of Observation \ref{STAR} that the sequence 
	$$0\lo \HH^d_{\fm}(M_1)\lo \HH^d_{\fm}(M)\lo \HH^d_{\fm}(M_2)\lo 0$$ splits.
		By Grothendieck's vanishing theorem, $\HH^{>d}_{\fm}(-)=0$. So, the functor  $\HH^{d}_{\fm}(-)$ is right exact. Also, it  preserves direct sums.
	  Due to the Watts' theorem, we know that
	 $$\HH^d_{\fm}(-)\cong (-)\otimes_R\HH^d_{\fm}(R) \cong (-)\otimes_RE_R(k).$$
	We combine  this  with the previous observation and conclude that $$\bigoplus_i M_i \otimes_RE_R(k) \cong M  \otimes_RE_R(k).$$ We  apply
	the Matlis functor to see $$\bigoplus_{i=1}^2\Hom_R\left( M_i \otimes_RE_R(k),E_R(k)\right) \cong\Hom_R\left( M \otimes_RE_R(k),E_R(k)\right).$$
Recall that  $\Hom_R(E_R(k),E_R(k))\cong\widehat{R}=R$. In the light  of tensor-hom adjunction we observe $$\oplus_i M_i^\ast\cong \oplus_i
	\Hom_R( M_i ,\Hom_R(E_R(k),E_R(k)) \cong\Hom_R( M ,\Hom_R(E_R(k),E_R(k))\cong M^\ast.$$
	Taking another $(-)^\ast$ yields that $M^{\ast\ast}\cong M_1^{\ast\ast}\oplus M_2^{\ast\ast}$.
		By reflexivity,  $$M\cong M^{\ast\ast}\cong M_1^{\ast\ast}\oplus M_2^{\ast\ast}\cong M_1\oplus M_2.$$  In view of Fact \ref{csp}
	$\zeta$ splits.
\end{proof}

\begin{discussion}\label{ausseq}i)
Let  $P_1\overset{\partial_1}{\rightarrow}P_0 \overset{\partial_0}{\to} M\to0$ be a finite projective presentation of $M$. The \emph{transpose} of $M$ is $\Tr M:=\coker ({\partial_1}^*)$.  There is a useful exact sequence:
$$
\Tor_2^R(\Tr\Syz_nM,N)\rightarrow\Ext^n_R(M,R)\otimes_RN\stackrel{f_n}\longrightarrow\Ext^n_R(M,N)
\rightarrow\Tor_1^R(\Tr\Syz_nM,N)\rightarrow0\quad(\ast)
$$

ii) Concerning Theorem \ref{high}, we present a replacement for Watts' theorem. Since $\pd(M)\leq 1$, $\Syz_1(M)$ is free. In particular, its presentation is given by $0\overset{\partial_1}{\rightarrow}P_0 \overset{\partial_0}{\to} \Syz_1M\to0$. By definition 
$\Tr\Syz_1(M)=0$, and so $$\Tor_2^R(\Tr\Syz_1(M),M_1)=\Tor_1^R(\Tr\Syz_1(M),M_1)=0.$$ In view of $(\ast)$, $\Ext^1_R(M,R)\otimes_RM_1\cong\Ext^1_R(M,M_1)$.

iii) One may control the error term $\coker(f_1)$ in a nontrivial  case. Suppose  $\pd(M)= 2$. Then $$\Tr\Syz_1(M)\cong\Ext^1_R(\Syz_1(M),R)\cong \Ext^2_R(M,R).$$ So,  $\coker(f_1)=\Tor_1^R(\Ext^2_R(M,R),M_1)$.
 \end{discussion}

One may like to conclude
the splitting of a short exact sequence of modules of projective dimension at most one,
from  splitting of the corresponding sequence of torsion modules. This is not the case:

\begin{example}
	Let $(R,\fm)$ be a $2$-dimensional regular local ring and note that $\zeta:=0\to R\to R^2 \to \fm\to 0$ is a short exact sequence of modules of projective dimension at most one such that the corresponding
	sequence of torsions splits, but $\zeta$ is not of splitting type.
\end{example}

Concerning Theorem \ref{high}, the presented bound on projective dimension  is optimal:

\begin{example}
	Let $(R,\fm,k)$ be a $3$-dimensional regular local ring and note that $\zeta:=0\to R\to R^2 \to \fm\to 0$ is a short exact sequence of modules of projective dimension at most two  such that 
$$0=\HH^{d-1}_{\fm}(M)\cong\HH^{d-1}_{\fm}(M_1)\oplus\HH^{d-1}_{\fm}(M_2)=0,$$  but $\zeta$ is not of splitting type. \end{example}\begin{proof}Indeed, we  use  $0\to \fm\to R\to k\to 0$ and the induced long
exact sequence $$0=\HH^{1}_{\fm}(k)\lo \HH^{2}_{\fm}(\fm)\lo \HH^{2}_{\fm}(R)=0$$ to conclude
that $\HH^{d-1}_{\fm}(M_2)=0$. Suppose on the way of contradiction $\zeta$  splits. Then $\fm$ is principal. This implies that  $R$ is 1-dimensional, which is excluded by the assumption.
\end{proof}

Concerning Theorem \ref{high}, one can not replace projective dimension with $G$-dimension:
\begin{example}
	Let $(R,\fm)$ be a $1$-dimensional Gorenstein ring which is not regular. We look at the  exact sequence  $$\zeta:=0\lo \Syz_1(\fm)\lo R^{\beta_0(\fm)} \lo \fm\lo 0$$  of finitely generated  modules of  $G$-dimension  zero such that the corresponding
	sequence of torsions splits. Suppose on the way of contradiction $\zeta$  splits. Then $\fm$ is free. This implies that  $R$ is regular, which is excluded by the assumption.
\end{example}

\begin{fact}\label{mc}(See \cite[Proposition 2.22]{v})
	Let $R$ be a commutative ring and $f:M^m\to M^n$ be a homomorphism of $R$--modules
	given by an $n \times m$ matrix $(a_{ij})$. Then $f$   is injective iff
	the ideal generated by  the minors of order $m$ does not annihilated a nonzero
	element of $M$.
\end{fact}

\begin{proposition}\label{p1}
	Let $(R,\fm)$ be a $1$-dimensional Cohen-Macaulay local ring with a canonical
	module $\omega_R$ and let $\alpha:=0\to M_1\to M \to M_2\to 0$ be a short exact sequence of finitely generated modules of projective dimension at most one. The following holds:
	\begin{enumerate}
		\item[a)] The sequence
		$ \zeta:=0\to M_1\otimes\omega_R\to M\otimes\omega_R \to M_2\otimes\omega_R\to 0$ is exact.
		\item[b)] If $\tor(M\otimes\omega_R)=\tor(M_1\otimes\omega_R)\oplus\tor(M_2\otimes\omega_R)$, then $\zeta$ splits.
	\end{enumerate}
\end{proposition}

\begin{proof}
	$a)$:
	Let $\xi:=0\to R^m\stackrel{f}\lo R^n\to M_2\to 0$ be a free resolution of $M_2$. Note that $f$ is represents by an $n \times m$ matrix $(a_{ij})$. Let $I$ be the ideal generated by  the minors of order $m$.  In the light of Fact \ref{mc}, $I$
	does not annihilated a nonzero
	element of $R$. Tensor $\xi$ with $\omega_R$ we have the following
	exact sequence	$$ 0\lo\Tor_1^R(M_2,\omega_R)\lo R^m\otimes\omega_R\stackrel{f\otimes1}\lo R^n\otimes\omega_R \lo M_2\otimes\omega_R\lo 0\quad(+)$$ The representing matrix of $f\otimes1$ is given by  $(a_{ij})$.
	Suppose on the way of contradiction that $I$ is  annihilated by a nonzero
	element of $\omega_R$. This in turns imply that $I$  annihilated a nonzero
	element of $\Hom_R(\omega_R,\omega_R)$. Since
	$\Hom_R(\omega_R,\omega_R)=R$ is faithful, we get to a contradiction. Thus,  $I$ is not  annihilated a nonzero
	element of $\omega_R$.  Thanks to Fact \ref{mc}, $f\otimes1$ is injective. In view of $(+)$ we see $\Tor_1^R(M_2,\omega_R)=0$. Tensor  $\alpha$ with $\omega_R$ we have the following
	exact sequence	$$ 0=\Tor_1^R(M_2,\omega_R)\lo M_1\otimes\omega_R\lo M\otimes\omega_R \lo M_2\otimes\omega_R\lo 0,$$ which yields part $a)$.
	
	b): Since $R$ is one-dimensional, and in view of Fact \ref{th}, $\HH^0_{\fm}(-)=\tor(-)$.
	We use this along with our assumption to see  $\HH^0_{\fm}(M\otimes \omega_R)\cong\oplus_i\HH^0_{\fm}(M_i\otimes \omega_R).$ 
	We apply the local
	duality theorem to see  $\Ext^1_R(M\otimes \omega_R,\omega_R)\cong\oplus_i\Ext^1_R(M_i\otimes \omega_R,\omega_R).$
	Tensor this  with  $M_2$, we obtain  $$\Ext^1_R(M\otimes \omega_R,\omega_R)\otimes_R M_1  \cong \oplus_i\Ext^1_R(M_i\otimes \omega_R,\omega_R)\otimes_RM_1.$$ Let $\mathcal{C}$
	be the family of modules of projective dimension at most one. Let $F:\mathcal{C}\to Ab$ be the functor
	defined by 
	$$F(L):=\Ext^1_R(M_2\otimes\omega_R,L\otimes\omega_R).$$
	Let $0\to R^m\to R^n\to L\to 0$ be a free resolution of $L$. We observed in part $a)$
	that  $$ 0 \lo R^m\otimes\omega_R \lo R^n\otimes\omega_R \lo L\otimes\omega_R\lo 0 $$
	is exact. Recall that $\id_R(\oplus\omega_R)<\infty$. From this, $\id_R(L\otimes\omega_R)<\infty$, and so $$\id_R(L\otimes\omega_R)=\depth_R(R)=1.$$
	This implies that the functor $F$ is right exact.  Also, it  preserves direct sums. In the light of Watts' theorem we see $$\Ext^1_R(M_1\otimes\omega_R,L\otimes\omega_R)=F(L)\cong F(R)\otimes_RL=\Ext^1_R(M_2\otimes\omega_R,\omega_R)\otimes_RL.$$  We apply this along with Fact \ref{csp}
	to deduce that the following natural sequence 
	$$0\lo \Ext^1_R(M_2\otimes\omega_R,M_1\otimes\omega_R)\lo \Ext^1_R(M\otimes\omega_R,M_1\otimes\omega_R)\lo \Ext^1_R(M_1\otimes\omega_R,M_1\otimes\omega_R)\lo 0,$$
	is exact. This shows that  $$\Hom_R(M\otimes\omega_R,M_1\otimes\omega_R)\twoheadrightarrow \Hom_R(M_1\otimes\omega_R,M_1\otimes\omega_R) \lo 0$$  is surjective.
	By looking at the preimage of $1\in\Hom_R(M_1\otimes\omega_R,M_1\otimes\omega_R)$ we deduce  $$ M\otimes\omega_R \cong\bigoplus _{i=1}^2 (M_i\otimes\omega_R),$$  and we get the desired claim.
\end{proof}

\section{Subfunctors of cohomology}

We start with

\begin{proposition}\label{ptri}
	Let $(R,\fm)$ be a $d$-dimensional Cohen-Macaulay complete local ring and let $\mathcal{F}$ be an additive functor
	which is a direct summand of $\HH^d_{\fm}(-)$ and  preserves direct sums. Then $\mathcal{F}$ is trivial.
\end{proposition}

\begin{proof}
	By Grothendieck's vanishing theorem, $\mathcal{F}$ is right exact.  According to the Watts' theorem, $$\mathcal{F}(\mathcal{L})=F(R)\otimes_R\mathcal{L}.$$
In the light of Matlis theory, $\HH^d_{\fm}(R)$ is indecomposable if and only if $\HH^d_{\fm}(R)^v$
	 is indecomposable. We apply the local duality theorem  to deduce that $\HH^d_{\fm}(R)^v=\omega_R$ which is indecomposable.
	Note that
	$\mathcal{F}(R)$ is a direct  summand of $\HH^d_{\fm}(R)$.
	Since  $\HH^d_{\fm}(R)$ is indecomposable, either $\mathcal{F}(R)=0$ or $\mathcal{F}(R)=\HH^d_{\fm}(R)$.
	We plug this in the previous observation to see either $\mathcal{F}(\mathcal{L})=0$ or   $$\mathcal{F}(\mathcal{L})=\mathcal{F}(R)\otimes_R\mathcal{L}=\HH^d_{\fm}(R)\otimes_R\mathcal{L}\cong \HH^d_{\fm}(\mathcal{L}),$$ i.e., $\mathcal{F}=0$ or $\mathcal{F}(-)=\HH^d_{\fm}(-)$. 
\end{proof}

The Cohen-Macaulay assumption is important:

\begin{example}
	Let $R:=\frac{\mathbb{Q}[[x_1,\ldots,x_4]]}{(x_1,x_2)\cap(x_3,x_4)}$. Then $\HH^2_{(x_1,x_2)}(-)\oplus\HH^2_{(x_3,x_4)}(-)\cong \HH^2_{\fm}(-)$ is a nontrivial
	decomposition of functors.
\end{example}

\begin{proof}
	Recall that $\HH^i_{0}(-)=0$ for all $i>0$.	By Mayer–Vietoris sequence, we have
	$$0=\HH^1_{(x_1,x_2)\cap(x_3,x_4)}(-) \to \HH^2_{(x_1,x_2)}(-)\oplus\HH^2_{(x_3,x_4)}(-)\to \HH^2_{\fm}(-)\to\HH^2_{(x_1,x_2)\cap(x_3,x_4)}(-)=0.$$Thus,
$$\HH^2_{\fm}(-)\cong	\HH^2_{(x_1,x_2)}(-)\oplus\HH^2_{(x_3,x_4)}(-).$$
	Due to Lichtenbaum–Hartshorne vanishing
	theorem, both of $\HH^2_{(x_1,x_2)}(-)$ and $\HH^2_{(x_3,x_4)}(-)$ are nonzero.
\end{proof}

Concerning Proposition  \ref{ptri}, the direct summand assumption is needed: 

\begin{example}
Let $(R,\fm)$ be a $d$-dimensional regular ring and let $\mathcal{F}(-):=\Ext^d_R(R/\fm,-)$. We left to the reader to check that $\mathcal{F}(-)\hookrightarrow  \HH^d_{\fm}(-)$ and that $\mathcal{F}$ is nontrivial.
\end{example}

Finding direct summand of $\HH^{\ast}_{\fm}(-)$  inspired from \cite{comment}, and has the following applications:

\begin{fact}(See \cite[Claim 4.1.A]{a})\label{a}
		Let $M$ be a finitely generated module over any commutative ring and let $\mathcal{L}$ be any module. Then $\Ext^i_R(M,\mathcal{L})$ has no nonzero projective submodule for all $i>0$.
\end{fact}

\begin{proposition}\label{idc}
	Let $(R,\fm)$ be a  complete  Gorenstein  local ring and $M$ be finitely generated. Suppose $\HH^i_{\fm}(M)$
	is nonzero and injective. Then	$i=\dim R$.	In particular, the following are equivalent:
	\begin{enumerate}
		\item[a)]  $M$  is  $(S_2)$,
		\item[b)] $M$ is free,
			\item[c)] $M$ is Cohen-Macaulay.
	\end{enumerate}
\end{proposition}

\begin{proof}
Let $d:=\dim R$. First, assume that $d-i>0$. By local duality,
$\Ext^{d-i}_R(M,R)$ is free. According to Fact \ref{a},  this is possible only if $\Ext^{d-i}_R(M,R)=0$.
In dual words, $\HH^i_{\fm}(M)=0$. But, this case excluded from the assumptions.
Hence,  $i=d$. Now, we prove  the particular case.
The only nontrivial implication is $a)\Rightarrow b)$.
Thus,  we assume that $M$ is $(S_2)$.
Recall that
$\HH^d_{\fm}(M)$
is injective.
In view of   Matlis  duality, $M^{\ast}$
is free. So,  $M^{\ast\ast}$ is as well. Since $M$ is $(S_2)$ it follows from \cite[Theorem 3.6]{Syz} that $M$ is reflexive.
Therefore, $M\cong M^{\ast\ast}$ is free.
\end{proof}

The $(S_2)$-condition is not a consequence of  the injectivity of the nonzero module $\HH^i_{\fm}(M)$. 
\begin{example}	Let $(R,\fm,k)$ be a quasi-Gorenstein ring of  dimension $d\geq2$. In view
	of $0\to \fm\to R\to k\to 0$ we see $\HH^d_{\fm}(\fm)=\HH^d_{\fm}(R)=E_R(k)$ is injective, but  $\fm$ is not free. As another example,
	look at the top local cohomology of $M:=R\oplus k$.
\end{example}

We conclude the following  easy facts from Proposition \ref{idc}.
\begin{corollary}\label{holo}
	Let $R:=k[[x_1,\ldots,x_d]]$ where $k$ is a field of zero characteristic  and let $M$ be a  finitely generated $R$-module.  Then $M$ is holonomic (or, more generally a $\Der(R,k)$-module) if and only if $M$  is free as an  $R$-module. 
\end{corollary}

\begin{proof}
	Without loss of the generality we may assume that $M$ is nonzero.
	Suppose  $M$ is holonomic. Let $r:=\depth_R(M)$. Recall that  $\HH^r_{\fm}(M)\neq 0$. Due to  a result of Lyubeznik \cite{lyu} we know
$$\id(\HH^r_{\fm}(M))\leq \dim(\HH^r_{\fm}(M))=0.$$ In the light of Proposition \ref{idc} we see $r=\dim R$.	By definition,  $M$  is  maximal Cohen-Macaulay. We conclude from  Auslander-Buchsbaum formula that $M$ is free. The reverse implication is always true.
\end{proof}

\begin{corollary}\label{PRIME}
	Let $R$ be a regular ring  of  prime characteristic,  and let $M$ be a  finitely generated $R$-module.  Then $M$ is $\mathcal{F}$-module  if and only if $M$  is projective as an  $R$-module. 
\end{corollary}

\begin{proof}
	Without loss of the generality we may assume that $R$ is local.
	Along the same lines as Corollary \ref{holo} we get the desired claim.
\end{proof}

\begin{remark}
We leave to the reader to formulate  Proposition \ref{idc} in the setting of Cohen-Macaulay rings.
\end{remark}

\section{Cohomological finiteness}

Recall that a module $\mathcal{M}$ is called cohomologically finite  if
$\Ext^i_R(\mathcal{M},R)$ is finitely generated as an $R$-module for all $i\geq 0$.
Here, we use some ideas of Bredon:

\begin{lemma}\label{lbr}
Let $(R,\fm)$ be a regular local ring of dimension $d>0$  and $\mathcal{M}$ be torsion-free and cohomologically finite.
Then $\Ext^d_R(\mathcal{M},N)=0$ for any finitely generated $R$-module $N$.
\end{lemma}


\begin{proof}
	First, we show that $\Ext^d_R(\mathcal{M},R)=0$.
	Let $\mathcal{T}:=\ker(\mathcal{M}\stackrel{f}\lo \mathcal{M}^{\ast\ast})$ and $F:=\im(f)$.
There is an exact sequence $$0\lo \mathcal{T}\lo \mathcal{M}\lo F\lo 0.$$
Since any submodule  of a torsion-free is torsion-free,  we see $\mathcal{T}$ is torsion-free.
   Let $r\in R$. Since $\mathcal{T}$ is torsion-free,
 there is an exact sequence $$0\lo \mathcal{T}\stackrel{r}\lo \mathcal{T}\lo \mathcal{T}/r\mathcal{T}\lo 0.$$ This induces
 $$\Ext^d_R(\mathcal{T},R)\stackrel{r}\lo\Ext^d_R(\mathcal{T},R)\lo \Ext^{d+1}_R(\mathcal{T}/r\mathcal{T},R)=0,$$ i.e.,
 $\Ext^d_R(\mathcal{T},R)$
 is divisible. 
 Since $R$ is of global dimension $d$, any submodule of a free module
	is of projective dimension at most $d-1$. Recall that any module of the form 
	$(-)^\ast$ is a submodule of a free module. Since $F\subset\mathcal{M}^{\ast\ast}$ we deduce
the following.	\begin{enumerate}
		\item[Fact]A):  One has $\pd_R(F)<d$.
	\end{enumerate}
In view of  $$0=\Ext^d_R(F,R)\to\Ext^d_R(\mathcal{M},R)\lo\Ext^d_R(\mathcal{T},R)\lo \Ext^{d+1}_R(F,R)=0\quad(\ast)$$
	we see $\Ext^d_R(\mathcal{T},R)$ is finitely generated, and recall that  it is
 divisible. These yield that $\Ext^d_R(\mathcal{T},R)=0$.
By another use of $(\ast)$, we get $\Ext^d_R(\mathcal{M},R)=0$.

Now, let $N$ be a finitely generated $R$-module. We look at $$0\lo\Syz_1(N)\lo R^n\lo N\lo 0.$$
This induces the following$$0=\Ext^d_R(\mathcal{M},R^n)\lo\Ext^d_R(\mathcal{M},N) \lo\Ext^{d+1}_R(\mathcal{M},\Syz_1(N))=0,$$ and consequently, $\Ext^d_R(\mathcal{M},N)=0$.
\end{proof}

\begin{proposition}\label{br}
	Let $(R,\fm)$ be a $PID$  which is not complete and $\mathcal{M}$ be torsion-free and cohomologically finite.
	Then $\mathcal{M}$ is finitely generated.
\end{proposition}

\begin{proof}
	Let $\mathcal{T}:=\ker(\mathcal{M}\stackrel{f}\lo \mathcal{M}^{\ast\ast})$ and $F:=\im(f)$. By definition,
	$\mathcal{T}^\ast=0$. Due to Fact \ref{lbr}.A),  $F$ is free and of finite rank (recall that $\mathcal{M}^\ast$ is finitely generated). According to the proof of Lemma \ref{lbr} we know$$\Ext^1_R(\mathcal{T},R)=\Ext^1_R(\mathcal{M},R)=0.$$
	By a result of Nunke \cite[Theorem 8.5]{n},
	$\mathcal{T}=0$. In view of the exact sequence $$0\lo \mathcal{T}\lo \mathcal{M}\lo F\lo 0,$$  we deduce   $\mathcal{M}$ is free and of finite rank.
\end{proof}

\begin{observation}
Let $(R,\fm)$ be a 1-dimensional complete local domain. Then $\mathcal{Q}(R)/R$ is cohomologically finite.	
\end{observation}

\begin{proof}
	By a result of Matlis, $\pd(\mathcal{Q}(R))=1$. Following $0\to R\to \mathcal{Q}(R)\to\mathcal{Q}(R)/R\to 0$,
	we observe that $\Ext^{>1}_R(\mathcal{Q}(R)/R,R) =0$.
	Also, $\Hom_R(\mathcal{Q}(R)/R,R)=0$. 
	Since $R$ is complete, 
	$\Ext^1_R(\mathcal{Q}(R),R)=0$. We put this in$$0=\Hom_R(\mathcal{Q}(R),R)\to \Hom_R(R,R)\to\Ext^{1}_R(\mathcal{Q}(R)/R,R)\to \Ext^{1}_R(\mathcal{Q}(R),R) =0,$$
	and conclude that $\Ext^{1}_R(\mathcal{Q}(R)/R,R)=R$. So, $\Ext^{\ast}_R(\mathcal{Q}(R)/R,R)$ is finitely generated (and free). By definition, $\mathcal{Q}/R$ is cohomologically finite.
\end{proof}
We recall that a ring is slender if for each countable family   $\{\mathcal{M}_n\}$ the natural map $$\bigoplus_n \Hom_R(\mathcal{M}_n,R)\stackrel{\cong}\lo \Hom_R(\prod_{n=1}^{\infty}\mathcal{M}_n,R)$$ is an isomorphism.
Here, we present a connection to set theory of rings.

\begin{corollary}
	Let $R$ be a discrete valuation domain of positive dimension. The following are equivalent:
	\begin{enumerate}
		\item[a)]  cohomologically finite and finite are equivalent notions  over torsion-free modules;
		\item[b)] $R$ is slender.
	\end{enumerate}
\end{corollary}

\begin{proof}
	$ a)\Rightarrow b)$: Suppose on the way of contradiction that $R$ is not slender.
	By \cite[Theorem 2.9]{ek} $R$ is complete. In view of Fact \ref{au} (see below), the fraction field
	of $R$ is cohomologically finite. It is well-known and easy to see that
	the fraction field of $R$ is not finitely generated. Since it is torsion-free, we get to a contradiction.
	
	$ b)\Rightarrow a)$: In view of \cite[Theorem 2.9]{ek},
	$R$ is not complete. In  particular, we are in the situation of Proposition \ref{br}
	to deduce a).
\end{proof}

\begin{definition}
	A module $\mathcal{M}$ is called cohomologically  of finite length   if
	$\Ext^i_R(\mathcal{M},R)$ is of finite length as an $R$-module  for all $i\geq 0$.
\end{definition}

\begin{proposition}
	Let $G$ be an abelian group. Then $G$ is cohomologically  of finite length  
	iff $\ell_{\mathbb{Z}}(G)<\infty$. 
\end{proposition}

\begin{proof}
Thanks to \cite[Proposition V.14.7]{br} we know that $G$ is finitely generated. By fundamental theorem of finitely generated abelian groups, $G\cong \mathbb{Z}^n\oplus \bigoplus_i \frac{\mathbb{Z}^{r_i}}{n_i\mathbb{Z}^{r_i}}$ for some $n, n_i,r_i\in \mathbb{N}_0$. Since
 $\Hom(G,\mathbb{Z}) $ is of finite length, it follows that $n=0$. In other words, $G$ is of finite length.
\end{proof}

\begin{definition}
	Let $(R,\fm)$ be local.	A module $\mathcal{M}$ is called H-finite  if
	$\HH^i_{\fm}(\mathcal{M})$ is artinian as an $R$-module for all $i\geq 0$.
\end{definition}

By $\mathcal{Q}(R)$ we mean  the fraction field  of a local domain $R$.

\begin{remark}
For example, any finite module is  H-finite, but not the converse. Indeed, let $R$  be a local domain $R$ of positive dimension.  Since 
$\HH^i_{\fm}(\mathcal{Q}(R))=0$ for all $i\geq 0$, it is non-finite and H-finite.
\end{remark}

	\begin{proposition}\label{hf}
		Let  $(R,\fm)$  be a complete quasi-Gorenstein ring and  $\mathcal{M}$ be reflexive and H-finite. Then $\mathcal{M}$
		is finite.
	\end{proposition}

	\begin{proof}
Let $d:=\dim R$. Let $L$  be an $R$-module. According to Watts' theorem, $$\HH^d_{\fm}(L)=\HH^d_{\fm}(R)\otimes_RL=E_R(k)\otimes_RL.$$
	Thanks to Matlis theory,   $\HH^d_{\fm}(\mathcal{M})^v$ is finitely generated.
		In the light of  tensor-hom adjunction we observe $$\HH^d_{\fm}(\mathcal{M})^v\cong\Hom_R( \mathcal{M} ,\Hom_R(E_R(k),E_R(k))\cong \mathcal{M}^\ast $$ is finitely generated. We apply another $(-)^\ast$ and use the
			 reflexivity assumption, to see $\mathcal{M}$ is finitely generated.
	\end{proof}

	\begin{remark} 
Concerning Proposition \ref{hf}, the first (resp. second) item, see below,  shows that the reflexivity (resp. H-finite) assumption  is important:
			\begin{enumerate}
	\item[a)]  Let $(R,\fm,k)$ be a  local ring of positive dimension. Then $E_R(k)$ is H-finite, and is not finite.
	\item[b)] Let $(R,\fm)$ be a discrete valuation domain which is not complete. Clearly, $\mathcal{M}:=\oplus_{\mathbb{N}}R$ is not finitely generated. It is easy to see $\mathcal{M}^\ast=\prod_{\mathbb{N}} R$. Since $R$ is Slender, $$\Hom_R(\prod_{\mathbb{N}} R,R)\cong \bigoplus_n \Hom_R( R,R)=\mathcal{M}.$$So, $\mathcal{M}$ is reflexive.
\end{enumerate} 	
\end{remark}

		\begin{proposition}
		Let  $(R,\fm)$  be a complete  Gorenstein ring. Then $\mathcal{M}$ is  H-finite iff  $\mathcal{M}$
		is cohomologically finite.
	\end{proposition}

	\begin{proof}Let $d:=\dim R$.
	The point is to establish the local duality theorem for   H-finite modules over complete  Gorenstein rings. Note that  $f^0_{\mathcal{M}}:\HH^d_{\fm}(\mathcal{M})^v\stackrel{\cong}\lo \mathcal{M}^\ast,$ as we observed in the previous proposition. Since $\HH^{d-i}_{\fm}(-)$
	and $\Ext^i_R(-,R)$ are zero over projective modules and all $i>0$, it turns out from an inductive argument that  $\HH^{d-i}_{\fm}(\mathcal{M})^v\cong\Ext_R^i( \mathcal{M} ,R)$. This completes the proof.
	\end{proof}

In general, H-finite and cohomologically finite are not the same.
For example, let $R$ be  a local  domain of positive dimension.
Then, $\HH^{i}_{\fm}(\mathcal{Q}(R))=0$, i.e., $\mathcal{Q}(R)$ is H-finite.
By a result of Jensen \cite{j} there are cases for which $\Ext^1_R(\mathcal{Q}(R),R)$ is a non-empty
direct sum of  $\mathcal{Q}(R)$, i.e.,  $\mathcal{Q}(R)$ is not cohomologically finite.
The next section talks more on this topic.

\section{Cohomologically  zero injective modules} 

	By  $\eSupp(-)$, we mean $\bigcup_i\Supp(\Ext^i_R(-,R)).$
	
\begin{definition}
	A module  $(-)$ is called cohomologically zero if $\eSupp(-)=\emptyset$.
\end{definition}

Recall that a finitely generated module is cohomologically zero iff it is zero. 
Let us connect to the previous section:
\begin{corollary}
	Let $(R,\fm)$ be a $PID$. Then
	$\mathcal{Q}(R)$ is cohomologically finite iff $R$ is complete.
\end{corollary}

To see the corollary, we may assume that it is of positive dimension, then it follows by Proposition \ref{br} and the following result of  
 Auslander  \cite[Page 166]{comment}:
 
\begin{fact}\label{au}
	Let  $(R,\fm)$  be a  complete local integral domain of positive dimension. Then
 $\mathcal{Q}(R)$ is  cohomologically zero.
\end{fact}

This result of Auslander can be extended  in  $3$ different directions:

\begin{observation}\label{ausg}
	Let $R$ be   any   integral domain of positive dimension which is not necessarily  noetherian. Suppose $R$ is complete with respect to  a nontrivial ideal $I$. Then
	$\Ext^i_R(\mathcal{Q}(R),R)=0$ for all $i$.
\end{observation}

\begin{proof} Let $F$ be a flat module with the property that $F\otimes_R R/I=0$.
	The desired claim follows  immediately from \cite[Theorem 4.4]{sch2}, where it were shown that $\Ext^i_R(F,M^{\widehat I})=0$ 	for any $M$.
  It remains to note that $\mathcal{Q}(R)$ is flat and
	that  $\mathcal{Q}(R)\otimes _RR/I=0$.
\end{proof}

\begin{remark}
Recall from \cite{mat2} that a module $\mathcal{M}$ is called strongly cotorsion if 	$\Ext^i_R(\mathcal{Q}(R),\mathcal{M})=0$ for all $i$. By this terminology, the previous observation says that $R$  is  strongly cotorsion.
\end{remark}

Let $(R,\fm) $ be a complete-local integral domain.
In   \cite[Corollary 2]{matcanada},  Matlis proved
 that $\Ext^1_R(E_R(k),R)\neq  0$ iff $\dim R=1$. This is true for another class of rings as item b) of the next result indicates:

\begin{proposition}\label{extk}
	Let $(R,\fm,k)$ be a local Gorenstein integral domain of dimension $d>0$. Then \begin{equation*}
	\Ext^i_R(E_R(k),R)\simeq \left\{
	\begin{array}{rl}
	\widehat{R} & \  \   \   \   \   \ \  \   \   \   \   \ \text{if } i=d\\
	0&\  \   \   \   \   \ \  \   \   \   \   \ \text{otherwise} 
	\end{array} \right.
	\end{equation*}
	In particular, 	the following properties are true for the non-finitely generated module $E_R(k)$:
		\begin{enumerate}
		\item[a)]  	$E_R(k)$ is cohomologically finite iff $R$ is complete,
		\item[b)]	$\Ext^1_R(E_R(k),R)\neq  0$ iff $\dim R=1$,
			\item[c)] $E_R(k)$ is not cohomologically zero,
			\item[d)]	$\eSupp(E_R(k))=\Spec(R) \neq \{\fm\}=\Supp(E_R(k))$.
	\end{enumerate}
\end{proposition}

\begin{proof}
	The injective resolution of $R$ is given by   $$0\to R\to  \oplus_{\Ht(\fq)=0}E_R(R/\fq)\to\ldots\to \oplus_{\Ht(\fq)=d-1}E_R(R/\fq)\to E_R(k)\to 0.$$ Recall that $\Hom_R(E_R(k),R)=0$ because the image any map is finitely generated and divisible. Now, we recall that  \begin{equation*}
	\Hom_R(E_R(k),E_R(R/\fp))\simeq \left\{
	\begin{array}{rl}
	\widehat{R} & \  \   \   \   \   \ \  \   \   \   \   \ \text{if } \fp=\fm\\
	0&\  \   \   \   \   \ \  \   \   \   \   \ \text{otherwise} 
	\end{array} \right.
	\end{equation*} 
	Thus, $$ 	\Ext^i_R(E_R(R/ \fp),R)=H^i(0\to\  0\to\ldots \to 0\to \widehat{R}\to 0).$$ This proves the first part.

 a): Note that	$\widehat{R}$ is finitely generated
	as an $R$-module iff $R$ is complete.
	
b): The existence of a finitely generated injective module implies that the ring is artinian.	

Items  c) and $d)$ are clear, because $d>0$.
\end{proof}

 Here, we deal with the corresponding property of $E_R(R/ \fp)$ for a middle $\fp$.

\begin{lemma}
	Let $(R,\fm,k)$ be a regular  local ring  of dimension $d>0$ and $\fp$ be a prime ideal of height $d-1$. There is a number ${\mu}$ (finite, or infinite) and a module $H$ isomorphic to $\widehat{R_{\fp}}$ such that \begin{equation*}
	\Ext^i_R(E_R(R/ \fp),R)\simeq \left\{
	\begin{array}{rl}
	\frac{	\widehat{R_{\fp}^{\mu}} }{H}& \  \   \   \   \   \ \  \   \   \   \   \ \text{if } i=d\\
	0&\  \   \   \   \   \ \  \   \   \   \   \ \text{otherwise} 
	\end{array} \right.
	\end{equation*}
\end{lemma}

\begin{proof}
	Let $X(i)=\{\fq\in\Spec(R):\Ht(\fq)=i\}$. The minimal  injective resolution of $R$ is given by  $$\zeta:=0\to R\to  \oplus_{\fq\in X(0)}E_R(R/\fq)\to\ldots\to\oplus_{\fq\in X(d-2)}E_R(R/\fq) \stackrel{f}\lo \oplus_{\fq\in X(d-1)}E_R(R/\fq)\stackrel{g}\lo E_R(k)\to 0.$$ Recall that $\Hom(E_R(k),R)=0$ because the image of any such map  at the same time is both finitely generated and divisible. Suppose $i>0$. Recall that$$\Hom_R(E_R(R/ \fp),E_R(R/ \fp))=\Hom_{R_{\fp}}(E_R(R/ \fp)_{\fp},E_R(R/ \fp)_{\fp})=\widehat{R_{\fp}}.$$ Also, there is a $\mu$ such that $E_R(R/ \fp)^v=\widehat{R_{\fp}^\mu}$ (see \cite[Page 2392]{sch}).
	Now, use the fact:  \begin{equation*}
	\Hom_R(E_R(R/ \fp),E_R(R/\fq))\simeq \left\{
	\begin{array}{rl}
	\widehat{R_{\fp}} & \  \   \   \   \   \ \  \   \   \   \   \ \text{if } \fp=\fq\\
	\widehat{R_{\fp}^{\mu}}&\  \   \   \   \   \ \  \   \   \   \   \ \text{if }  \fq=\fm
	\\
	0&\  \   \   \   \   \ \  \   \   \   \   \ \text{otherwise} 
	\end{array} \right.
	\end{equation*}
	In particular, $$	\Ext^i_R(E_R(R/ \fp),R)=H^i\left(0\to\  0\to\ldots \to 0\to 	\widehat{R_{\fp}}\stackrel{h}\lo \widehat{R_{\fp}^{\mu}}\to 0\right)\quad(+)$$
	
	Let $g:=(g_{\fq})$ where $g_{\fq}:E_R(R/ \fq)\to E_R(R/ \fm)$. Suppose on the way of contradiction that $g_{\fq}=0$. Then $E_R(R/ \fq)\subset\ker(g)=\im(f)$. Let $f_1$ be the composition map $$\oplus_{\fp\in X(d-2)}E_R(R/\fp) \stackrel{f}\lo \oplus_{Q\in X(d-1) }E_R(R/Q)\twoheadrightarrow \oplus_{Q\in X(d-1) \setminus{\fq}}E_R(R/Q).$$ We look at 
	$$	\oplus_{\fq\in X(d-2)}E_R(R/\fq) \stackrel{f_1}\lo \oplus_{\fq\in X(d-1) \setminus{\fq}}E_R(R/\fq)\stackrel{g_1}\lo E_R(k)\to 0,$$ and observe that
	
		\begin{enumerate}
		\item[a)]  $\im(f_1)\oplus E_R(R/\fq)=\im(f)$ and
		\item[b)]	 $\ker(g)=\ker(g_1)\oplus E_R(R/\fq)$.
	\end{enumerate}

	Since
	$\ker(g)=\im(f)$ we deduce that $\ker(g_1)=\im(f_1)$.
	This is in contradiction with the minimality of $\zeta$. In sum, $g_{\fq}\neq 0$.
	
	Since  $g_{\fp}:E_R(R/ \fp)\to E_R(R/ \fm) $ is nonzero, we know $h:=\Hom(E_R(R/ \fp),g_{\fp})$ is nonzero.
	Let $A$ be a ring and $F:A\to \oplus_{i\in I} A$ be nonzero. We composite this with a suitable projection to the $i$-th component, $F$ induces a nonzero map  $$A\stackrel{F}\lo  \oplus_{i\in I} A\twoheadrightarrow A,$$and denote the composite map by $f_i$. In particular, $f_i$ is a multiplication by  some $r_i\in R$. By assumption, $\widehat{R_{\fp}}$
	is an integral domain. From these, we deduce that $h$ is injective.
	
	Let $H:=\im(h)$.
	Since $h$ is injective, it is isomorphic  to $\widehat{R_{\fp}}$, and in view of $(+)$ we deduce that  
	\begin{equation*}
	\Ext^i_R(E_R(R/ \fp),R)\simeq \left\{
	\begin{array}{rl}
	\frac{	\widehat{R_{\fp}^{\mu}} }{H}	& \  \   \   \   \   \ \  \   \   \   \   \ \text{if } i=d\\
	0&\  \   \   \   \   \ \  \   \   \   \   \ \text{otherwise} 
	\end{array} \right.
	\end{equation*}
This is what we want to prove.
\end{proof}

\begin{corollary}
Adopt the above notation and suppose in addition $R/ \fp$ is not complete. Then  $E_R(R/ \fp)$
is not cohomologically  zero.
\end{corollary}

\begin{proof}
 By a result of Schenzel \cite[Proposition 4.3]{sch} we know $\mu>1$.
 Since 	$$\Ext^d_R(E_R(R/ \fp),R)=	{	\widehat{R_{\fp}^{\mu}} }/{H}$$  and $H\cong\widehat{R_{\fp} }$, we deduce that $\Ext^d_R(E_R(R/ \fp),R)$ is not zero.
\end{proof}
Now, we state three corollaries on projective dimension of certain injective modules.

\begin{corollary}
	Let $(R,\fm)$ be a regular  local ring  of dimension $d>0$ and $\fp$ be a prime ideal of height $d-1$ such that $R/ \fp$ is not complete. Then  $\pd_R(E_R(R/ \fp))=d$.
\end{corollary}

\begin{proof}
	We observed in the previous corollary  that	$\Ext^d_R(E_R(R/ \fp),R)$ is not zero. Since global-dimension
	of $R$ is $d$ we get the desired claim.
\end{proof}

\begin{observation}\label{oi}(After Matlis)
	Let $(R,\frak{m})$ be a  Gorenstein local ring. Then $\pd(E_R(R/\fm))\leq \dim R$.
\end{observation}

\begin{proof}  Let $d:=\dim R$ and $\underline{x}=x_1,
	\dots, x_d$ be a system of parameters for $R$. Recall that the \v{C}ech complex of $R$ with respect to $\underline{x}$
	has the form $$\check{C}:=0\longrightarrow R \longrightarrow \oplus R_{x_i} \longrightarrow \dots \longrightarrow \oplus
	R_{x_1\dots \hat x_i \dots x_d}\longrightarrow R_{x_1 \dots x_d} \longrightarrow 0,$$ and it provides a flat resolution
	for the $R$-module $\text{H}_{\frak{m}}^d(R)=E_R(R/\fm)$.
	Thus,
	it has finite projective dimension. It turns out that
	$\pd(E_R(R/\fm))\leq \dim R$.
\end{proof}

Let us show that this is sharp:

\begin{corollary}
	Let $(R,\fm,k)$ be a regular  local ring  of dimension $d$. Then  $\pd_R(E_R(k))=d$.
\end{corollary}

\begin{proof}
We may assume $d>0$.	Recall from  Proposition  \ref{extk} that	$\Ext^d_R(E_R(k),R)\neq0$. Since global-dimension
	of $R$ is $d$ we get the desired claim.
\end{proof}
In Observation \ref{oic}, see below, we will replace the above regularity assumption with the   complete and Gorenstien assumption.
\begin{corollary}
	Let $(R,\fm)$ be a regular  ring  and $\fp$ be a prime ideal. Then  $\pd_R(E_R(R/ \fp))\geq \Ht(\fp)$.
\end{corollary}

\begin{proof}
It is easy to see that $\pd_{R}(\mathcal{M})\geq \pd_{R_\fp}(\mathcal{M}_\fp)$ for any $R$-module $\mathcal{M}$ and any $\fp$. Since $(E_R(R/ \fp))_{\fp}=E_{R_\fp}(k(\fp))$, the desired claim is in the previous corollary.
\end{proof}

\section{The realization problem}
We start with
\begin{definition}	(Jensen-Nunke) Let $i>0$.
	Recall from \cite{j2} that a module $\mathcal{M}$ is called $i$-realizable if there are modules $\mathcal{N},\mathcal{L}$
	such that $\mathcal{M}\cong \Ext^i_R(\mathcal{N},\mathcal{L})$. When $i:=\dim R<\infty$ we  say  $\mathcal{M}$ is realizable.\end{definition}

The following
result is not new. It was proved by Jensen \cite[Theorem 2]{j2} via some spectral sequences and derived functors $\vpl^{(i)}$. The following proof is  elementary.

\begin{corollary}\label{cr}
	Let $(R,\fm)$ be a   regular of positive dimension   $d$    and $M$ be finitely generated. Then
	$\widehat{M}\cong \Ext^d_R(E_R(k),M)$. In particular,
	if $M$ is complete then $M$ is realizable.
\end{corollary}

\begin{proof}
	Let $(-)$ be a finitely generated module and set $F(-):=\Ext^d_R(E_R(k),-)$. Since injective dimension
	of any module is at most $d$, $F(-)$ is right exact. Also, it preserves finite direct sum.
	By Watts' theorem, $$F(-)\cong F(R)\otimes_R (-)$$ for any finitely generated module.
	Recall from  Proposition \ref{extk} that $F(R)=\widehat{R}$. By plugging this in Watts' isomorphism, we get that $$\widehat{M}\cong M\otimes_R\widehat{R}\cong \Ext^d_R(E_R(k),M),$$as claimed.
\end{proof}

\begin{remark}
  If  we allow the module $\mathcal{N}$ in the above definition to be finitely generated, the story
	of Corollary \ref{cr}        will changes, see Fact \ref{a}.\end{remark}
	
\begin{example} The finitely generated assumption of $M$ in Corollary \ref{cr} is important:
	Let $R$ be a discrete valuation domain which is not a field. Then $\pd_R(\mathcal{Q}(R))=1$.
	There is a module $\mathcal{M}$ such that $\Ext^1_R(\mathcal{Q}(R),\mathcal{M})\neq 0$. Let $d\mathcal{M}$ be the largest divisible submodule of $\mathcal{M}$. In view of
	\cite[Theorem 7.1]{n}, $\frac{\mathcal{M}}{d\mathcal{M}}$ is not realizable.
\end{example}

Recall from \cite{mat2} that a module is called corosion if  $\Hom_R(\mathcal{Q}(R),\mathcal{M})=\Ext^1_R(\mathcal{Q}(R),\mathcal{M})=0$.
\begin{example}(Matlis)
Any cotorsion module $\mathcal{M}$ over an integral domain is $1$-realizable.
\end{example}

\begin{proof}
We apply  $\Hom_R(-,\mathcal{M})$ to $0\to R\to \mathcal{Q}(R)\to \frac{\mathcal{Q}(R)}{R}\to0$ and observe $$0=\Hom_R(\mathcal{Q}(R),\mathcal{M})\to \Hom_R(R,\mathcal{M})\lo \Ext^1_R(\mathcal{Q}(R)/R,\mathcal{M})\lo\Ext^1_R(\mathcal{Q}(R),\mathcal{M})=0,$$i.e., $\mathcal{M}\cong \Hom_R(R,\mathcal{M})\cong \Ext^1_R(\mathcal{Q}(R)/R,\mathcal{M}).$
\end{proof}

In order to extend  Corollary \ref{cr}, we need to state the following:

\begin{definition}	 Let $i>0$.
	We say a module $\mathcal{M}$ is called  homologically realizable at level $i$ if there are modules $\mathcal{N},\mathcal{L}$
	such that $\mathcal{M}\cong \Tor^R_i(\mathcal{N},\mathcal{L})$. When $i:=\dim R<\infty$ we  say  $\mathcal{M}$ is homologically  realizable.
 \end{definition}

\begin{fact}\label{toreal}(Matlis)
Let $R$ be a  domain.	Any torsion module  $\mathcal{M}$  is  homologically  realizable at level one. In particular, over $1$-dimensional rings a module is  homologically  realizable iff it is torsion.
\end{fact}

\begin{proof}
Since $\mathcal{M}$ is torsion, $\mathcal{M}\otimes_R\mathcal{Q}(R)=0$. Apply $-\otimes_R\mathcal{M}$
to $0\to R\to \mathcal{Q}(R) \to \frac{\mathcal{Q}(R)}{R}\to 0$ to deduce
$$ 0=\Tor^{R}_1(\mathcal{Q}(R),\mathcal{M})\to \Tor^{R}_1( {\mathcal{Q}(R)}/{R},\mathcal{M}) \to  R\otimes_R\mathcal{M} \to \mathcal{M}\otimes_R\mathcal{Q}(R)=0,$$because $\mathcal{Q}(R)$ is flat. So, $\mathcal{M} \cong\Tor^{R}_1({\mathcal{Q}(R)}/{R},\mathcal{M})$.
\end{proof}

\begin{proposition}\label{rtor}
	Let $(R,\fm)$ be a  local ring , $\underline{x}:=x_1,\ldots,x_n$ be a regular sequence
	and let $ \mathcal{A}$ be artinian.
	Then $ \mathcal{A}$ is homologically realizable at level $i$ for each $0<i\leq n$. In fact,
	$ \mathcal{A}\cong \Tor^R_i(\HH^i_{\underline{x}_i}(R), \mathcal{A})$ where $\underline{x}_i:=x_1,\ldots,x_i$.
\end{proposition}

 The following proof works for $\fm$-torsion modules.
 
\begin{proof}
Let $\fa_i:=(x_1,\ldots,x_i)$, and recall the natural isomorphism  $\HH^0_{\fa_i}( \mathcal{A})= \mathcal{A}$. This holds, because $ \mathcal{A}$  is an	$\fm$-torsion module.
Denote the $\check{C}heck$ complex of $R$  with respect to $\fa_i$ by $\check{C}(x_1,\ldots,x_i;R)$.
Since $\underline{x}$ is a regular sequence, it follows that $\check{C}(x_1,\ldots,x_i;R)$ 
is a flat resolution of $\HH^i_{\fa_i}(R)$, and consequently, it can be used for commuting tor-modules. Namely,
$$\Tor^R_j(\HH^i_{\fa_i}(R), \mathcal{A})\cong\HH_{j}(\check{C}(x_1,\ldots,x_i;R)\otimes_R \mathcal{A})\cong \HH^{j-i}_{\fa_i}( \mathcal{A}).$$
 Thus, $$ \Tor^R_i(\HH^i_{\fa_i}(R), \mathcal{A})\cong \HH^{0}_{\fa_i}( \mathcal{A})\cong  \mathcal{A},$$as claimed.
\end{proof}

\begin{corollary}
	Let $(R,\fm)$ be a  Cohen-Macaulay local ring of positive dimension and let $\{ \mathcal{A}_j\}$ be a directed family of artinian modules.
	Then $ \mathcal{A}:=\vil_j  \mathcal{A}_j$ is homologically realizable.
\end{corollary}

\begin{proof}Let $d:=\dim R$.
In view of Proposition \ref{rtor}, we know $ \Tor^R_d(\HH^d_{\fm}(R), \mathcal{A}_j) \cong  \mathcal{A}_j.$ Since $\Tor$-functors commutes with directed limits, we have $$ \Tor^R_d(\HH^d_{\fm}(R), \mathcal{A})\cong \vil\Tor^R_d(\HH^d_{\fm}(R), \mathcal{A}_j) \cong \vil  \mathcal{A}_j= \mathcal{A}.$$This is what we want to prove it.
\end{proof}

Let us give some examples that can't be written as a direct limit of artinian:

\begin{example}
	Let $I$ be any ideal. Then the conormal module $\frac{I}{I^2}$ is homologically realizable at level one.
There is a choice for $I$ such that the conormal module is not of finite length. Indeed, use  the natural isomorphism $\Tor^{R}_1(\frac{R}{I},\frac{R}{J})\cong \frac{I \cap J}{IJ}$.
\end{example}

Let $M,N $ be finitely generated  modules. The notation
$P(M, N)$ stands  for the set of  all $M\to N$ which   factor through free modules. By   stable-hom,  we mean
 $\underline{\Hom}_R(M,N) :=\frac{\Hom_R(M, N)}{P(M, N)}$.

\begin{example}(Auslander)
	The stable-hom    is homologically realizable at level one. Indeed, $$\underline{\Hom}_R(M,N)=\coker(f:\Hom_R(M,R)\otimes_RN\to\Hom_R(M,N)).$$ Recall from  Discussion \ref{ausseq} that
	$\coker(f)=\Tor_1^R(\Tr M,N)$.
\end{example}

Now, we  present a simple proof of \cite[Proposition 3]{j2} by Jensen:

\begin{proposition}\label{j3}
	Let  $(R,\fm)$  be a  complete local ring,  $\underline{x}:=x_1,\ldots,x_n$ be a regular sequence
	and let $M$ be finitely generated. Then $M\cong \Ext^i_R(\HH^i_{{\underline{x}_i}}(R),M)$  for each $0<i\leq n$. In particular,
$M$ is  realizable at level $i$. 
\end{proposition}

\begin{proof}
 Let $\fa_i:=(x_1,\ldots,x_i)$ and $ \mathcal{A}:=M^v$. Then $ \mathcal{A}$ is artinian. Thanks to Proposition \ref{rtor} we know
	$$ \mathcal{A}\cong \Tor^R_i(\HH^i_{\fa_i}(R), \mathcal{A}).$$
	By taking another Matlis duality and using Matlis theory we observe that
$$
	M \cong M^{vv} = \mathcal{A}^v \cong\Tor^R_i(\HH^i_{\fa_i}(R),M^v)^v 
	\cong\Ext^i_R(\HH^i_{\fa_i}(R),M^{vv}) 
	\cong\Ext^i_R(\HH^i_{\fa_i}(R),M),$$
as claimed.
\end{proof}

\begin{corollary}\label{ci}
	Let $(R,\fm)$ be a complete Cohen-Macaulay local ring     of  dimension  $d>0$    and let $M$ be finitely generated.
	Then $M\cong \Ext^{d}_R(\HH^{d}_{\fm}(R),M)$. In particular,
	$M$ is  realizable.
\end{corollary}

\begin{proof}
It is enough to note that $\HH^{d}_{\fm}(R)\cong \HH^{d}_{\underline{x}}(R)$, where $\underline{x}$ is a full parameter sequence.
\end{proof}

\begin{observation}\label{oic}
	Let $(R,\frak{m})$ be a complete Gorenstein local ring. Then $\pd(E_R(R/\fm))=\dim R$.
\end{observation}

\begin{proof}  Recall from Observation \ref{oi} that
	$\pd(E_R(R/\fm))\leq \dim R$.
	Recall from Corollary \ref{ci}  that $ \Ext^{d}_R(E_R(R/\fm),R)\cong R\neq 0$. So, $\pd(E_R(R/\fm))=\dim R$.
\end{proof}The following result extends a result of Matlis \cite[Theorem 6]{matcanada} from 1-dimensional case to any dimension. Also, a Cohen-Macaulay assumption from \cite[Corollary 3.13]{matne}.
\begin{corollary}\label{mat2}
	Let $(R,\fm)$ be a  local ring. Then      $R$ is Gorenstein iff $\pd(E_R(R/\fm))<\infty$.
\end{corollary}

\begin{proof}
One side direction is in Observation \ref{oic}, it is easy to remove the complete assumption.
Now, suppose $t:=\pd(E_R(R/\fm))<\infty$. Let $$0\lo F_t\lo \ldots\lo F_0\lo E_R(R/\fm)\lo 0$$ be a free resolution of $E_R(R/\fm)$. Apply Matlis duality, gives us the following injective resolution $$0\lo \widehat{R}\lo F_0^v\lo\ldots \lo F_t^v\lo 0. $$We see $i:=\id_R(\hat{R})<\infty$.
Following Cartan-Eilenberg (see \cite[VI.4.1.3]{ce}) we have $$0=\Ext^j_R(R/\fm, \hat{R})\cong\Ext^j_{\hat{R}}(R/\fm\otimes_R\hat{R}, \hat{R})\cong\Ext^j_{\hat{R}}(\hat{R}/\fm_{\hat{R}} , \hat{R})$$for all $j>i$. 
From this, $\id_{\hat{R}}(\hat{R})<\infty$, i.e.,  $\hat{R}$ is Gorenstein. Consequently, $R$  is Gorenstein.
\end{proof}

The non Cohen--Macaulay locus of an excellent ring $R$ is $$\nCM( R):=\{\fp\in\Spec(R):R_{\fp}\textit{ is not Cohen--Macaulay}\}$$ which is a closed subset of $\Spec(R)$ with respect to the Zariski topology. So, there is an ideal $I$ such that $ nCM( R)=\V(I)$ and we call it the ideal of definition.

\begin{corollary}
	Let $(R,\fm)$ be excellent with the ideal of definition $I$. Then $\pd_R(E_R(R/\fp))= \infty$
	for all $\fp\in\V(I)$.
\end{corollary}

\begin{proof}
	Let  $\fp\in\V(I)$. Then $R_{\fp}$
	is not Cohen-Macaulay. In particular, it is not Gorenstein. By Corollary \ref{mat2}
	$\pd_{R_{\fp}}(E_{R_{\fp}}(R_{\fp}/\fp R_{\fp}))=\infty$. Since $\pd_R(E_R(R/\fp))\geq  \pd_{R_{\fp}}(E_{R_{\fp}}(R_{\fp}/\fp R_{\fp}))$ we conclude that $\pd_R(E_R(R/\fp))= \infty$. 
\end{proof}Here, we collect some example of realizable  flat module over non-artinian rings that are not finitely generated:
\begin{example} 
\begin{enumerate}
\item[i)] Let  $(R,\fm)$  be a Gorenstein local ring. Then $\widehat{R}$ is realizable.
\item[ii)] Let  $(R,\fm)$  be a complete Cohen-Macaulay local ring and let $\{M_j\}$ be a family of finitely generated modules. Then $\prod_{j}M_j$ is $i$-realizable for any $0<i\leq \dim R$. 
\item[iii)] Let  $R$ be as item ii). Then $\prod_{\mathbb{N}}R$ is realizable.
\end{enumerate}
\end{example}

\begin{proof}
	Let $d:=\dim R$.\begin{enumerate}
	\item[i)] In view of Proposition \ref{extk} we know that $\widehat{R}\cong\Ext^d_R(\HH^d_{\fm}(R),R)$. By definition, $\widehat{R}$ is realizable.	
	\item[ii)] For simplicity, we assume $i=d$.
	By the above corollary we know that $M_j=\Ext^d_R(\HH^d_{\fm}(R),M_j)$.
	Thus, $\prod_{j}M_j\cong\Ext^d_R(\HH^d_{\fm}(R),\prod_{j}M_j)$, and so  $\prod_{j}M_j$ is realizable. 
	\item[iii)] This is a special case of ii)
\end{enumerate}
\end{proof}

The following result was proved by Jensen \cite[Proposition 5]{j}. 
	Here, we present an elementary proof:

\begin{fact}\label{bf}
	Let $R$ be a complete local ring and $\{M_i\}$ be an inverse system of finitely generated modules. Suppose $\id_R(M_i)\leq n$ for all $i$. Then $\id_R(\vpl M_i) \leq n$.
\end{fact}

\begin{proof}
	Dualizing the injective resolution   $0\to M_i\to E_0\to\cdots\to E_n\to0$,  yields the following flat resolution
	$$0\lo E_n^v\lo\cdots\lo E_0^v\lo M_i^v\lo 0.$$ Since
	$\Tor$ commutes with directed limit, it follows that $\fd_R(\vil M_i^v) \leq n$.
	Let $$0\lo F_n\lo\cdots\lo F_0\lo \vil (M_i^v)\lo 0$$ be a flat resolution. Then
	$$0\lo(\vil M_i^v)^v\lo F_0^v\lo \cdots\lo  F_n^v\lo 0$$ is an injective resolution.
	Recall that $$(\vil M_i^v)^v=\vpl M_i^{vv}=\vpl M_i.$$ This completes the proof.
\end{proof}

In fact, one may use Fact \ref{bf} and show that:

\begin{fact}
	Let $R$ be a complete local ring and $\{ \mathcal{A}_i\}$ be an inverse system of Matlis reflexive  modules  (e.g. artinian modules). Suppose $\id_R( \mathcal{A}_i)\leq n$ for all $i$. Then $\id_R(\vpl  \mathcal{A}_i) \leq n$.
\end{fact}

\section{Homologically zero flat modules}

We recall the following elementary result of Auslander:

\begin{fact}(See \cite[Proposition 4.3]{comment})
Let $(R,\fm)$ be a complete local, $M$ be finitely generated. Then $\Ext^p_R(\vil \mathcal{A}_i , M)\cong\vpl\Ext^p_R(\mathcal{A}_i,M)$ for any directed family $\{\mathcal{A}_i\}$.
\end{fact}

We simplify  the following result of Jensen and Buchweitz-Flenner: 

\begin{corollary}(See \cite[Corollary 1]{bf})\label{f}
 If $(R,\fm)$ is a complete local noetherian ring and $\mathcal{F}$ is a flat $R$-module then
	$\Ext^p_R(\mathcal{F}, M) = 0$ for all $p \geq 1$ and all finite $R$-modules $M$.
\end{corollary}

\begin{proof}Due to Lazard's theorem we know that any flat module $\mathcal{F}$ is a direct limit of free modules $\{F_i\}$.
By the above fact,  $$\Ext^p_R(\mathcal{F} , M)=\Ext^p_R(\vil F_i , M)\cong\vpl\Ext^p_R(F_i,M)=0,$$as claimed. 
\end{proof}
As for as I know, the history of the next result comes back to the beautiful  red book of Kaplansky \cite[theorem 19]{kap54}:

\begin{corollary}(Compare with \cite[Theorem 7.1]{n})\label{n1}
	Let $(R,\fm)$ be a $PID$. The following are equivalent:
	\begin{enumerate}	
		\item[a)] $\Ext^1_R(\mathcal{F},N)=0$   for  any torsion-free  module $\mathcal{F}$ and any finitely generated $R$-module $N$,
		\item[b)] $\Ext^1_R(\mathcal{F},N)$ is  finitely generated   for  any torsion-free  module $\mathcal{F}$ and any finitely generated $R$-module $N$,
		\item[c)]	$\Ext^1_R(\mathcal{F},R)$ is  finitely generated   for  any torsion-free  module $\mathcal{F}$,
		\item[d)]	$\Ext^1_R(\mathcal{F},R)=0$  for  any countably generated torsion-free  module $\mathcal{F}$,
		\item[e)]	$\Ext^1_R(\mathcal{F},R)=0$  for  any torsion-free  module $\mathcal{F}$,	\item[f)]  	$R$ is complete,
		\item[g)]  	$R$ is realizable,
		\item[h)]  	 any finitely generated module is realizable.
	\end{enumerate}
\end{corollary}

\begin{proof}
	$a)\Rightarrow b)\Rightarrow c)$:  These are trivial.
	$c)\Rightarrow d)$: This is in Lemma \ref{lbr}.
	$d)\Rightarrow e):$ This is trivial. $e)\Rightarrow f):$
	Recall that   flat and  torsion-free are the same notions. Now, use \cite[Theorem 1]{j2}. $f)\Rightarrow g):$ See Corollary \ref{cr}.
	$g)\Rightarrow h):$ It follows that any finitely generated module is complete. Now, see Corollary \ref{cr}. $h)\Rightarrow a):$ This implies that $R$ is complete, and recall that torsion-free are the same notions. Now, use Corollary \ref{f}.
\end{proof}

Here, we collect   some examples of cohomogically (non-) zero flat modules from literature.

\begin{example}(Gruson)
	Let $k$ be an uncountable field, and  let $R:=k[X,Y]_{(X,Y)}$. 
Then $\Ext^i_R(\mathcal{Q}(R),R)=0$ iff $i\neq 2$.	In particular, $\Ext^2_R(\mathcal{Q}(R),R)$ is not finitely generated.
\end{example}

\begin{proof} 
	Suppose on the way of contradiction that  $\Ext^2_R(\mathcal{Q}(R),R)$ is  finitely generated.
	By a result of Gruson \cite[Proposition 3.2]{Gru} we know that $R$ is complete with respect to adic topology of the collection $S:=\{rR:r\in R\}$, i.e., the map $f$ in the following exact sequence
	$$0\lo R\stackrel{f}\lo \vpl_{r\in S} \frac{R}{rR}\lo \Ext^1_R(\mathcal{Q}(R),R) \lo 0$$ is an isomorphism. In other words, $\Ext^1_R(\mathcal{Q}(R),R)=0$. Also, $\Hom_R(\mathcal{Q}(R),R)=0$. In other words, $\mathcal{Q}(R)$ is cohomologically finite.
	In view of	Lemma \ref{lbr} we see $\Ext^2_R(\mathcal{Q}(R),R)=0$. Since $\id_R(R)=2$, we have $\Ext^i_R(\mathcal{Q}(R),R)=0$ for all $i$. According to \cite[Theorem 9.19]{j3}, this yields $\mathcal{Q}(R)=0$.
	This is a contradiction.
\end{proof}

The following is now immediate:
\begin{corollary}
	(Kaplansky) Adopt the previous notation. Then $\pd_R(\mathcal{Q}(R))=2$.
\end{corollary}
The uncountable assumption on $k$ is really needed, as the next example says:

\begin{example} 
	Let $(R,\fm)$ be a countable Gorenstein integral domain. Then $\Ext^i_R(\mathcal{Q}(R),R)=0$ iff $i\neq 1$.
\end{example}

\begin{proof} 
	Recall that $\pd_R(\mathcal{Q}(R))=1$ and $\Hom_R(\mathcal{Q}(R),R)=0$. It remains to
	recall from \cite[Theorem 9.18]{j3} that over a  countable  local Gorenstein ring, a   flat  module is zero iff it is cohomologically zero.
\end{proof}
\begin{example}\label{2com}
	Let $k$ be any field, and  let $R:=k[[X,Y]]$. 
	Then $\pd_R(\mathcal{Q}(R))=2$.
\end{example}

\begin{proof}
It follows by countable prime avoidance for complete rings (see \cite{sv}) that there are uncountable family of height one prime ideals which are principal, as $R$ is $UFD$. It remains to apply the proof of \cite[Theorem 2]{kapc}.
\end{proof}
Let us compute
$\pd_R(\mathcal{Q}(R))$ in some singular cases. First, we present the following easy observation:
\begin{discussion}
	\label{surjective}
	Let $(R,\fm,k)$ be a $2$-dimensional   Cohen-Macaulay local integral domain containing the uncountable field $k$.   Let $x,y$ be an $R$-sequence.	According to a result of Hartshorne \cite[Proposition 1]{H66}, $A:=k[x,y]$ is the polynomial ring and the inclusion map  $A\hookrightarrow R$ is a flat extension. In the light of \cite[Theorem 2]{kapc} we observe that
	$\pd_A(\mathcal{Q}(A))= 2$. 
	Let $N$ be any $R$-module. So, it is equipped with the structure of an $A$-module. By definition, $\Ext_A^j(\mathcal{Q}(A),N)=0$ for all $j>2$.
	Following Cartan-Eilenberg (see \cite[VI.4.1.3]{ce}) we have $$0=\Ext^j_A(\mathcal{Q}(A), N)\otimes_AR\cong\Ext^j_{{R}}(\mathcal{Q}(A)\otimes_AR,N)\cong\Ext^j_{{R}}(\mathcal{Q}(A)\otimes_AR, N)$$for all $j>2$. We deduce from this that $$\pd_R(\mathcal{Q}(A)\otimes_AR)\leq 2.$$\end{discussion}

\begin{example}\label{89}
	Let $R:=k[x^2,xy,x^2]_{\fm}$.
	Then $\pd_R(\mathcal{Q}(R))\leq 2.$
\end{example}

\begin{proof}
	Let $A:=k[x^2,y^2]$. In view of 
	Discussion \ref{surjective}
	$$\pd_R(\mathcal{Q}(A)\otimes_AR)\leq 2\quad(\ast)$$
	For simplicity, we set $U:=x^2$, $V:=y^2$ and $W:=xy$. Then $A:=k[U,V]$ and $R:=\frac{k[U,V,W]}{(UV-W^2)}$. Let us rewrite
	$R:=A[W]/(W^2-a)$ where $a:=UV\in A$. It is easy to see that $$\mathcal{Q}(R)=\mathcal{Q}(A)[W]/(W^2-a),$$
	because $w(wa^{-1})=w^2a^{-1}=1$, i.e., the inverse of $w$ is $wa^{-1}\in \mathcal{Q}(A)[W]/(W^2-a)$. Now, recall that
	\begin{equation*}
	\begin{array}{clcr}
	\mathcal{Q}(R) &=\mathcal{Q}(A)[W]/(W^2-a)\\
	&
	= \mathcal{Q}(A)\otimes_A(A[W]/(W^2-a))\\
	&= \mathcal{Q}(A)\otimes_AR.
	\end{array}
	\end{equation*}
	In the light of $(\ast)$,
	$\pd_R(\mathcal{Q}(R))\leq 2.$ 
\end{proof}

\begin{example}\label{891}
  Let $R:	= \mathbb{C}[ x,y,z ]_{\fm}/(x^2 + y^ 3 + z^5 )$. Then $\pd_R(\mathcal{Q}(R))=2$.
\end{example}

\begin{proof}	
The method of Example \ref{89} shows that
	$\pd_R(\mathcal{Q}(R))\leq 2.$ Suppose
	on the way of contradiction that $\pd_R(\mathcal{Q}(R))\neq2$. Then
	$\pd_R(\mathcal{Q}(R))=1$. Following
	\cite[Theorem 1]{kapc} $\mathcal{Q}(R)$
	is countably generated as an $R$-module.
	Recall that a noetherian local ring with uncountable residue field equipped with  countable prime avoidance.
It follows by this that there are uncountable family of height one prime ideals. Recall that the singular ring $R$ is $UFD$. The ring is equipped with uncountable family of  irreducible elements. So, $\mathcal{Q}(R)$ is not countably generated. This contradiction shows that $\pd_R(\mathcal{Q}(R))=2$.
\end{proof} 

\begin{remark}
i)	In the final section we present a modern proof of  Discussion \ref{surjective} and Example \ref{89} via applying the machinery invented by Gruson and Ranaud \cite{GR}.
Despite this, the above arguments are elementary, independent from \cite{GR} and spiritual.

ii)
 In Theorem \ref{113} we present the general case of Example \ref{891}. In fact, the mentioned theorem inspired by Example \ref{891}.
\end{remark}

\begin{example}\label{schex1}
	Let $(R,\fm)$ be  complete  and $f\in\fm$ be regular. Then  $R_{f}$ is homologically zero flat module.
\end{example}

\begin{proof} 
	It is easy to see $\Hom_R(R_f,R)=0$. By Corollary \ref{f} $\Ext^+_R(  R_f,R)=0$.
\end{proof}
We need the reverse part of Example \ref{schex1}:

\begin{fact}\label{schpr}(See \cite[Theorem 1.1]{sch}) 
	Let $x_1,\ldots, x_d$ be a full system of parameters for $(R,\fm)$. Then $\Ext^1_R(\oplus R_{x_i},R)=0$ iff $R$ is complete in $\fm$-adic topology.
\end{fact}

\section{Splitting with the theme of Kaplansky}
We  start with the following splitting criteria:
 
\begin{proposition}\label{splitmat}
	Let $(R,\fm)$ be a local domain and $f\in R$ be nonzero.
	Let $\zeta:=0\to R\to \mathcal{A} \stackrel{g} \lo  R_f\to 0$ be such that $\mathcal{A}$  decomposes into nonzero modules.
	Then $\zeta$ splits.
\end{proposition}

\begin{proof}
We may assume that $f$ is not invertible. Let $\mathcal{M}$  be a module of positive grade with respect to $fR$.
The flat resolution of $\frac{R_f}{R}$ is given by $$0\lo R\lo R_f\lo  \frac{R_f}{R}\lo 0.$$ Then $$\Tor_1^R(\frac{R_f}{R},\mathcal{M})=H^0(\Check{C}(f,R))\otimes_R\mathcal{M}=\HH^0_{(f)}(\mathcal{M})=0\quad(\ast)$$
Let $\mathcal{A}\cong \mathcal{A}_1\oplus \mathcal{A}_2$ be a nontrivial decomposition. Recall that  $$\frac{R_f}{R}\otimes_RR_f=\frac{R_f}{R_f}=0.$$ We apply $\frac{R_f}{R}\otimes_R-$  to $\zeta$ to deduce that 	$$0=\Tor_1^R(\frac{R_f}{R},R_f) \lo \frac{R_f}{R}\lo \frac{R_f}{R}\otimes_R \mathcal{A}\lo \frac{R_f}{R}\otimes_RR_f=0,$$i.e., $\frac{R_f}{R}\cong \frac{R_f}{R}\otimes_R\mathcal{A}$.
		\begin{enumerate}
		\item[Claim] A): $\frac{R_f}{R}$ is indecomposable.
		\item[]  Indeed, thanks to $\Check{C}heck$-complex, we know that $$\frac{R_f}{R}=H^1(\Check{C}(f,R))=\HH^1_{(f)}(R).$$ It is 
		easy to see  $$\End_R(\HH^1_{(f)}(R))\cong R^{\widehat{}_{(f)}}$$ which is a commutative integral domain. From this,  $\frac{R_f}{R}$ is indecomposable.
	\end{enumerate}Now, recall that $\mathcal{A}\cong \mathcal{A}_1\oplus \mathcal{A}_2$. Apply  $\frac{R_f}{R}\otimes_R-$
to it and use $\frac{R_f}{R}\cong \frac{R_f}{R}\otimes_R \mathcal{A}$ to see
$$\frac{R_f}{R}\cong \frac{R_f}{R}\otimes_R \mathcal{A}\cong  \frac{R_f}{R}\otimes_R \mathcal{A}_1\oplus \frac{R_f}{R}\otimes_R \mathcal{A}_2.$$
We combine Claim A) with this, to assume without loss of the generality that $\frac{R_f}{R}\otimes_R \mathcal{A}_1=0$. We apply $-\otimes_R\mathcal{A}_1$ to $0\to R\to R_f\to  \frac{R_f}{R}\to 0$ and deduce that $$0\stackrel{(\ast)}=\HH^0_{(f)}(\mathcal{A}_1)\cong\Tor_1^R(\frac{R_f}{R},\mathcal{A}_1)\lo \mathcal{A}_1\lo  R_f \otimes_R \mathcal{A}_1\lo \frac{R_f}{R}\otimes_R\mathcal{A}_1=0,$$i.e., $\mathcal{A}_1\cong  R_f \otimes_R \mathcal{A}_1$. Recall that  $g: \mathcal{A}\to R_f$.
Then
\begin{equation*}
\begin{array}{clcr}
R_f &=g(\mathcal{A})\\
&
= g(\mathcal{A})_f\\
&= g(\mathcal{A}_1\oplus \mathcal{A}_2)_f \\
&=g(\mathcal{A}_1)_f\oplus g(A_2)_f\\
&= g((\mathcal{A}_1)_f)\oplus g(\mathcal{A}_2)_f\\
& =g( \mathcal{A}_1) \oplus g(\mathcal{A}_2)_f.
\end{array}
\end{equation*}

  Since  $R_f$ is indecomposable, either $g( \mathcal{A}_1)=0$ or
$  g(\mathcal{A}_2)_f=0$. Suppose on the way of contradiction that $g( \mathcal{A}_1)=0$. On the one hand $\mathcal{A}_1\subset \ker(g)=R$,
and on the other hand $\mathcal{A}_1$ is not finitely generated. This contradiction implies that $  g(\mathcal{A}_2)_f=0$. Thus, 
$g(\mathcal{A}_1)=R_f$. From this, $R+\mathcal{A}_1=\mathcal{A}$. Since $R\cap \mathcal{A}_1=0$ we have $R\oplus \mathcal{A}_1=\mathcal{A}$.
The desired claim follows by this.
\end{proof}

\begin{question}(Matlis+Kaplansky, \cite{k})
	For what integral domains is it true that any torsionfree module of rank two is a direct sum of modules of rank one?
\end{question}

Here, we   reprove  \cite[Theorem 61]{mat3} and an essential part of \cite{k}.

\begin{corollary} Let $(R,\fm)$ be a  local domain such that every
torsion-free $R$-module of rank $2$ that is not finitely generated is a direct
sum of modules of rank $1$. Then $R$ is complete in the  $\fm$-adic topology.
\end{corollary}

\begin{proof}
	Let $x_1,\ldots, x_d$ be a system of parameters of $\fm$.	
	The assumptions  allow us to apply
 Proposition \ref{splitmat}, and deduce that $\Ext^1_R( R_{x_i},R)=0$ for all $i$.
	By Fact \ref{schpr}, $R$ is complete in $\fm$-adic topology.
\end{proof}

\begin{remark}
	Among other things, \cite{moh2} talks about rings  for which
any finitely generated reflexive module is a direct
sum of modules of rank $1$.
 \end{remark}

The following simplifies \cite[Theorem 3.4]{Chase} and removes its integral domain assumption:

\begin{proposition}\label{24}
  Let $(R,\frak{m})$  be a 1-dimensional complete ring. The following are equivalent: 
  
 \begin{enumerate}	
 	\item[i)] $t(M) \oplus \overline{M} \cong M$ for all finitely generated $R$-module, \item[ii)] $R$ is regular. \end{enumerate}
\end{proposition}
\begin{proof}
$i) \Rightarrow ii)$:	Suppose $\Ext_{R}^{2} (R/\frak{m},R/\frak{m}) \neq 0 $. By applying $\Hom(-,R/\frak{m})$ to
	$$ 0 \longrightarrow \frak{m} \longrightarrow R \longrightarrow R/\frak{m}\lo 0$$ 
	 yields us that
	$$0=\Ext_{R}^{1}(R,R/\frak{m}) \longrightarrow \Ext_{R}^{1}(\frak{m},R/\frak{m}) \longrightarrow \Ext^{2}( R/\frak{m} , R/\frak{m}) \longrightarrow \Ext^{2}(R,R/\frak{m})=0.$$Now, recall Yoneda's correspondence:
	$$\Ext_{R}^{1}(\frak{m},R/\frak{m}) \rightleftharpoons \lbrace 0 \longrightarrow R/\frak{m} \longrightarrow X \longrightarrow \frak{m} \longrightarrow 0  \rbrace.$$
	Since $\Ext_{R}^{1}(\frak{m},R/\frak{m})$ is not zero, then by Yoneda lemma there exists an exact sequence
	$$0\longrightarrow R/\frak{m} \longrightarrow X \longrightarrow \frak{m} \longrightarrow 0 \quad(\ast),$$
	which is not split.
	Since $\lbrace \frak{m}, R/\frak{m} \rbrace $ are finitely generated, then $X$ is finitely generated. Recall that
	$\frak{m}$ is torsion-free. Then  \begin{enumerate}
		\item[$\bullet$] $\tor(X)=R/\frak{m}$ and 	
		\item[$\bullet$] $X/ \tor(X)$ is $\fm$. \end{enumerate}
	 From these, and according to the assumption, $t(X) \oplus \overline{X} \cong X$. In view of Theorem 1.1, we get to the contradiction that $(\ast)$ splits. So, $\Ext_{R}^{2} (R/\frak{m},R/\frak{m}) = 0 $, and consequently, $R$ is regular.
	 
	 $ii) \Rightarrow  i)$:
	 This is trivial.
\end{proof}

\section{A question by  Gerstner}
	Let $k$ be an uncountable field, and  let $R_0:=k[X,Y]_{(X,Y)}$. 
We observed in \S 8 that $\Ext^i_{R_0}(\mathcal{Q}(R_0),R_0)=0$ for $i\leq 1$.
Note that $\depth(R_0)=2$.
Now, let $R$ be any commutative ring. Gerstner asked: 

 \begin{question}\label{101}
 	For which property of $R$ to
 	ensure, that condition
 $\Hom_R(M,R)=\Ext^1_R(M,R)=0$
 	on $R$-module $M$ implies $M =0$. 
 \end{question}

One may deal with finitely generated modules, see \cite[Proposition  3]{oo}. 
\begin{proposition}
Let  $(R,\fm)$ be local, and assume the restriction that we focus on finitely generated modules. Then
in 	Question \ref{101} the desired  property completely determined. Namely, $\depth(R)\leq 1$.
\end{proposition}

\begin{proof}
Suppose $\depth(R)\geq 2$. There is an $R$-regular sequence two, namely $x,y$. Set $M:=R/(x,y)$.
Recall that $$\grade(I,R)=\inf\{i:\Ext^i_R(R/I,R)\neq 0\}.$$ From this,
 $\Hom_R(M.R)=\Ext^1_R(M,R)=0$ but $M$ is not free.
 
  Conversely, suppose that $\depth(R)\leq 1$. Let $M$ be such that  $\Hom_R(M,R)=\Ext^1(M,R)=0$. Suppose  $\depth(R)=0$ (resp. $\depth(R)=1$).
 Since $M^\ast$ is free (resp. $\Ext^1_R(M,R)=0$), and in view of \cite[Lemma 2.6]{ram} (resp. \cite[Lemma 3.3]{dao}) we know $M$ is free. In particular, $M$ is  reflexive. Since $M\cong M^{\ast\ast}=0^\ast=0$,
 we conclude that $M$ is zero, as claimed by 	Question \ref{101}. 
\end{proof}

\section{Mores on $\pd_{R}(\mathcal{Q}(R))$}

We start by the following nontrivial fact from \cite{GR}:

\begin{fact}\label{GR}
Let $(R,\fm)$ be local and $F$ be flat.
Then $\pd(F)\leq \dim R$.	
\end{fact}

\textbf{Second proof of Discussion
\ref{surjective} and
Example \ref{89}}: Apply the above fact. $\Box$

\begin{notation}
	The continuum hypothesis abbreviated by $\textbf{CH}$. By definition, it means  $2^{\aleph_{0}}=\aleph_1$.
\end{notation}

\begin{theorem}\label{113}(Also, see \cite[3.3.2]{GR})
	Let $(R,\fm,k)$ be a complete local integral domain and of
	dimension $d$. The following assertions are valid:\begin{enumerate}	
		\item[a)] Suppose $d=1$. Then $\pd_{R}(\mathcal{Q}(R))=1$.
		\item[b)] If $d=2$ and $R$ is $UFD$, then $\pd_{R}(\mathcal{Q}(R))=2$.
		\item[c)] Adopt $\textbf{CH}$ and $k:=\mathbb{C}$. Then
		 $\pd_{R}(\mathcal{Q}(R))\leq2$.
		\item[d)]	Assume in addition to c) that  $R$ is $UFD$ and of dimension bigger than one. Then $\pd_{R}(\mathcal{Q}(R))=2$.
	\end{enumerate}
\end{theorem}
\begin{proof}a) This is due to Matlis. As another argument apply Fact \ref{GR}.

b) Recall that $UFD$ rings are normal domains. By Serre's characterization of normality (see \cite[Theorem 23.8]{mat}), $R$ satisfies $(S_2)$ and $(R_1)$. Since $d=2$, the $(S_2)$ condition implies that $R$ is Cohen-Macaulay. Also, recall from \cite{sv} that a noetherian complete local ring  equipped with  countable prime avoidance. Now, the reminder of the argument is similar to Example \ref{891} and we left the straightforward modification to the reader.

c) Without loss of generality we may and do assume that $d>2$.	By Cohen's structure theorem (see \cite[Theorem 29.4 (iii)]{mat}),
	there is a regular local ring $A$
	of dimension $d$
	such that $A\subseteq R$ is finite. In particular, it is integral.
	Let $S:=A\setminus \{0\}$. Then $S^{-1}A
\subseteq S^{-1}R$ is integral. Recall from \cite[Lemma 9.1]{mat} that if in the integral extension of rings one of the ring is field, then it follows that the other one is field. From this, $S^{-1}R$ is a field, and so it is equal to the fraction field of $R$. We proved that 
$$\mathcal{Q}(R)=\mathcal{Q}(A)\otimes_AR.$$
	Let us apply the $\textbf{CH}$
	assumption 
along with the \cite[Theorem 6.10]{os} to observe that
$\pd_A(\mathcal{Q}(A))= 2$ (recall that $d>2$). 
Let $N$ be any $R$-module. So, it is equipped with the structure of an $A$-module. By definition, $\Ext_A^j(\mathcal{Q}(A),N)=0$ for all $j>2$. Since $\mathcal{Q}(A)$ is flat over $A$, we deduced that $\Tor^A_+(\mathcal{Q}(A),R)=0$.
Following Cartan-Eilenberg (see \cite[VI.4.1.3]{ce}) we have $$0=\Ext^j_A(\mathcal{Q}(A), N)\otimes_AR\cong\Ext^j_{{R}}(\mathcal{Q}(A)\otimes_AR,N)\cong\Ext^j_{{R}}(\mathcal{Q}(A)\otimes_AR, N)$$for all $j>2$. We deduce from this that $$\pd_R(\mathcal{Q}(A)\otimes_AR)\leq 2.$$
So, $$\pd_R(\mathcal{Q}(R))\leq 2.$$

d) The previous item shows that
$\pd_R(\mathcal{Q}(R))\leq 2.$ Suppose
on the way of contradiction that $\pd_R(\mathcal{Q}(R))\neq2$. This implies that
$\pd_R(\mathcal{Q}(R))=1$. In the light of
\cite[Theorem 1]{kapc} we observe that $\mathcal{Q}(R)$
is countably generated as an $R$-module.
Let $\{1/x_n:n\in\mathbb{N}\}$ be a generating set of $\mathcal{Q}(R)$. Without loss of generality, and by an easy induction on the number of irreducible components of $x_n$,  we may assume in addition that $\{x_n\}$ are irreducible.
Recall from \cite{sv} that a noetherian complete local ring  equipped with  countable prime avoidance. 
Let $$\Sigma:=\{\fp\in\Spec(R):\Ht(\fp)=1\}.$$
Suppose on the way of contradiction that $\Sigma$ is countable. Since $\fm\subset \bigcup _{\fp\in\Sigma}\fp$, we deuce that $\fm=\fp$ for some $\fp\in\Sigma$.  Also, recall that $\dim R>1$.
This contradiction shows that there are uncountable family of height one prime ideals. Now, we are going to use the assumption that $R$ is $UFD$, and conclude that its height prime ideals are principal. In sum, the ring is equipped with uncountable family of  irreducible elements. So, there is an irreducible $x$ such that $x\notin\{x_n\}$ even up to generating ideal. There are $r_i\in R$ and $n\in\mathbb{N}$ such that $$1/x=\sum_{i=1}^n\frac{r_i}{x_i}.$$Let $u_i:=(\prod_j x_j)/x_i$. Then $$\prod_j x_j =\sum_i( r_i u_i)x
\subseteq (x).$$Since $(x)$ is prime, $(x_j)=(x)$ for some $j$. This  contradiction shows that $\pd_R(\mathcal{Q}(R))=2$.
	\end{proof}

\begin{theorem}
	Let $R$ be an affine integral domain of
	dimension $d$ over $\mathbb{C}$. The following assertions are valid:
\begin{enumerate}	
	\item[a)] Suppose $d=1$. Then $\pd_{R}(\mathcal{Q}(R))=1$.
	\item[b)] If $d=2$ and $R$ is $UFD$, then $\pd_{R}(\mathcal{Q}(R))=2$.
	\item[c)] Adopt $\textbf{CH}$. Then $\pd_{R}(\mathcal{Q}(R))\leq2$.
	\item[d)]	Assume in addition to c) that  $R$ is $UFD$ and of dimension bigger than one. Then $\pd_{R}(\mathcal{Q}(R))=2$.
\end{enumerate}
\end{theorem}

\begin{proof}
	This is similar to the previous theorem. Instead of Cohen's structure theorem, we need to apply the Noether's normalization theorem (see \cite[Lemma 33.2]{mat}), and 
	recalling that a noetherian local ring with an uncountable residue field equipped with  countable prime avoidance.
\end{proof}
\begin{example}	Let $R:=\mathbb{C}[x_1,\ldots,x_5]/(x_1^2+\ldots+x_5^2)$. 
	 If   $\textbf{CH}$ holds, then $\pd_{R}(\mathcal{Q}(R))=2$.	
\end{example}

\begin{proof}
Recall that   $R$ is 
	Cohen-Macaulay
	and of dimension four.
		 Since $R$ is $UFD$, and in view of the previous theorem, $\pd_{R}(\mathcal{Q}(R))=2$.
\end{proof}
 
 In order to see the assumptions from Theorem \ref{113} are (not) important, we ask:
 \begin{question} Let $R$ be a domain.
 	
 i)	Suppose $\depth(R)=1$ and $\dim(R)=2$. What is $\pd_{R}(\mathcal{Q}(R))$?

ii) Suppose $\depth(R)=\dim(R)=2$. What is $\pd_{R}(\mathcal{Q}(R))$? \end{question}

\section{Splitting with the theme of Matlis}
The following extends Matlis result (see Corollary 12.2) to higher rank:
\begin{theorem}
	Let $(R,\fm)$ be a 1-dimensional complete domain with fraction filed $Q$. Let $B$ and $B'$ be torsion-free and of finite rank. Suppose $H^1_\fm(B)\cong H^1_\fm(B')$. Then $B\oplus Q^{t'}= B'\oplus Q^{t}$. In particular, if they have same rank then $B\cong B'$.
\end{theorem}

\begin{proof}
	By \cite{matcanada} there are finitely generated modules $M,M'$ and two integers $t,t'$ such that $B\cong Q^t\oplus M$ and  $B'\cong Q^{t'}\oplus M'$. Since $H^1_\fm(Q^{t'})=H^1_\fm(Q^{t})$, we see
 	$H^1_{\fm}(M')=H^1_\fm(M)$. By local duality, $$\Hom(M',\omega_R)=H^1_\fm(M')^v=H^1_\fm(M)^v=\Hom(M,\omega_R).$$Since $M\subseteq B$ it is torsion-free. Since the ring is one-dimensional, $M$ is maximal Cohen-Macaulay. The same thing holds for $M'$.  
Now, we use \cite[3.3.10]{BH} to conclude that 
 $$\begin{array}{ll}
M&\stackrel{}=\Hom( \Hom(M^{},\omega_R),\omega_R) \\
&\stackrel{}= \Hom( \Hom(M^{'},\omega_R),\omega_R)\\
&\stackrel{}=M^{'}.\\
\end{array}$$So, $$B'\oplus Q^{t}\cong Q^{t+t'}\oplus M'\cong Q^{t+t'}\oplus M=B\oplus Q^{t'},$$and the claim follows.
\end{proof}

\begin{corollary}(Matlis)
	Let $(R,\fm)$ be a 1-dimensional complete domain with  fraction filed $Q$. Let $B\subset Q$ and $B'\subset Q$ both congaing $R$. If $Q/B\cong Q/ B'$ then $B\cong B'$.
\end{corollary}

\begin{proof}
	This is clear as $Q/B=H^1_\fm(B)$.
\end{proof}

Here, is the higher-dimensional version:
\begin{corollary} 
	Let $(R,\fm)$ be a d-dimensional complete domain with  fraction filed $Q$. Let $B_1\subset Q$ and $B_2\subset Q$ both containing $R$. If $Q/B_1\cong Q/ B_2$ then $B_1\cong B_2$.
\end{corollary}

\begin{proof}
	Following \cite[Corollary 6.11(2)]{mat2}, one may deduce that $B_i$ are cotorsion. This allows us to apply \cite[Proposition 3.2(2)]{mat2} to see 	
 
	$$\begin{array}{ll}
	B_2	&\cong\Hom_{R}  (Q/R, Q/R\otimes B_2)\\
	&\cong\Hom_{R}  (Q/R, Q/  B_2)\\
	&\cong\Hom_{R}  (Q/R, Q/  B_1)\\
&\cong\Hom_{R}  (Q/R, Q/R\otimes B_1)\\	
&\cong B_1,\end{array}$$
 as claimed.
\end{proof}


\begin{thebibliography}{99}
 
\bibitem{a}M. Asgharzadeh,
\emph{Homological subsets of  Spec}, arXiv:1701.03001.

 \bibitem{moh2}
M. Asgharzadeh, \emph{Reflexivity revisited}, arXiv:1812.00830.

\bibitem{comment} M. Auslander, \emph{Comments on the functor Ext}, Topology {\bf{8}} (1969), 151-166.

\bibitem{br}Glen E.
Bredon,  \emph{Sheaf theory}. Second edition. Graduate Texts in Mathematics, {\bf{170}}, Springer-Verlag, New York, 1997.

\bibitem{BH}{W. Bruns and J. Herzog}, {\it Cohen-Macaulay rings},  Cambridge Studies in Advanced Mathematics,
{\bf 39},
Cambridge University Press, Cambridge, 1993.

\bibitem{b}Jan-Erik
Björk,  \emph{Analytic D-modules and applications}, Mathematics and its Applications, {\bf{247}}. Kluwer Academic Publishers Group, Dordrecht, 1993. 

\bibitem{bf}H.O.
Buchweitz and H. Flenner, \emph{Power series rings and projectivity}, Manuscripta Math. {\bf{119}} (2006),  107–114.



\bibitem{ce}
H. Cartan and S. Eilenberg, \emph{Homological Algebra},
Princeton University Press, Princeton, N. J. (1956).

\bibitem{Chase}Stephen U.
Chase, \emph{ Direct products of modules}. Trans. Amer. Math. Soc. {\bf{97}} (1960), 457–473.

\bibitem{dao}H.
Dao,    M. Eghbali and  J. Lyle, \emph{Hom and Ext, revisited},
J. Algebra {\bf{571}} (2021), 75–93.

\bibitem{ek}P.C.
Eklof and   A.H. Mekler, \emph{Almost free modules. Set-theoretic methods}, Revised edition. North-Holland Mathematical Library, {\bf{65}}, North-Holland Publishing Co., Amsterdam, (2002).

\bibitem{Syz}E.
Evans and  P. Griffith, \emph{Syzygies}, LMS Lecture Note Series, {\bf{106}} Cambridge University Press, Cambridge, 1985.

\bibitem{j}C.U.
Jensen,  \emph{$On \Ext^1_R(A,R)$ for torsion-free A}. Bull. Amer. Math. Soc. {\bf{78}} (1972), 831–834.

\bibitem{j2}C.U.
Jensen,   \emph{ On the vanishing of $\vpl^{(i)}$}, J. Algebra  {\bf{15}}, 151–166 (1970).


 \bibitem{j3}C.U. Jensen,   \emph{Les foncteurs derives de $\vpl$ et leurs applications en theorie des modules}, LNM   {\bf{254}},
 Springer-Verlag, Berlin-New York, 1972.

\bibitem{dd}
Hotta  Ryoshi, Takeuchi  Kiyoshi, Tanisaki  Toshiyuki,  \emph{D-modules, perverse sheaves, and representation theory}, Translated from the 1995 Japanese edition by Takeuchi. Progress in Mathematics, {\bf{236}}. Birkhäuser Boston, Inc., Boston, MA, 2008.

\bibitem{lyu}
G. Lyubeznik,  \emph{Finiteness properties of local cohomology modules (an application of
D-modules to commutative algebra)},  Invent. Math. {\bf113} (1993), 41--55.

\bibitem{lyu2}
G. Lyubeznik,  \emph{F-modules: applications to local cohomolgy and D-modules in characteristic $p> 0$},  J. Reine Angew. Math {\bf491} (1997), 65-130.


\bibitem{oo}Otto Gerstner, \emph{
Reflexivity Properties of Rings},
Publications mathématiques et informatique de Rennes (1975)
Issue: 2, page 1-4.
\bibitem{H66}
R. Hartshorne, \emph{A property of  A-sequences}, Bull. Soc. Math. France {\bf94} (1966), 61-65. 
\bibitem{g}
A. Grothendieck, \emph{Local Cohomology}, Notes by R. Hartshorne, Lecture Notes in Math.,  {\bf20}, Springer, 1966.

\bibitem{Gur}R.M.
Guralnick,   \emph{Lifting homomorphisms of modules},
Illinois J. Math. {\bf29} (1985), 153–156.

 \bibitem{Gru} L. Gruson, \emph{Dimension homologique des modules plats sur an anneau commutatif noetherien,} in: Convegno di Algebra Commutativa, INDAM, Rome, 1971, in: Symp.
Math., vol. XI, Academic Press, London, 1973, pp. 243–254.
	

\bibitem{GR}M. 	Raynaud and L. Gruson,  \emph{Critères  de platitude et de projectivité. Techniques de "platification'' d'un module}, Invent. Math.  {\bf{13}} (1971), 1–89. 


\bibitem{kap54} Irving Kaplansky,  \emph{Infinite abelian groups}, University of Michigan Press, Ann Arbor, 1954.

\bibitem{k} Irving
Kaplansky, \emph{
Decomposability of modules},
Proc. Amer. Math. Soc.  {\bf13} (1962), 532–535.

\bibitem{kapc} Irving Kaplansky, \emph{
The homological dimension of a quotient field}, Nagoya Math. J. {\bf27} (1966),
139-142.

\bibitem{matl}
E. Matlis, \emph{Injective modules over noetherian rings}, Pacific  J. Math. {\bf8} (1958), 511–528.

\bibitem{mat2}
E. Matlis, \emph{Cotorsion modules},
Mem. Amer. Math. Soc. {\bf49} (1964).

\bibitem{mat3}
E. Matlis, \emph{Torsion-free modules},  Chicago Lectures in Mathematics. The University of Chicago Press, Chicago-London, 1972.


\bibitem{matne}
E. Matlis,\emph{
The Koszul complex and duality},
Comm. Algebra {\bf1} (1974), 87-144.



\bibitem{matcanada}
E. Matlis, \emph{Some properties of Noetherian domains of dimension one}, 
Canadian J. Math. {\bf13} (1961), 569–586.

\bibitem{mat}
H. Matsumura, \emph{Commutative Ring Theory}, Cambridge Studies in Advanced Math, \textbf{8}, (1986).

\bibitem{n}R. J. Nunke, \emph{Modules of extensions over Dedekind rings}, Illinois J. Math.   {\bf3} (1959),  222–241.

\bibitem{ram}
M. Ramras, \emph{Betti numbers and reflexive modules}, Ring theory (Proc. Conf., Park City, Utah, 1971),   297~-308. Academic Press, New York, 1972.

\bibitem{os} B. L. Osofsky, \emph{Homological dimension and the continuum hypothesis},
Trans. Amer. Math. Soc. {\bf132} (1968), 217–230.

\bibitem{sch} P.
Schenzel,  \emph{A criterion for completeness}, Proc. AMS {\bf143} (2015),  2387–2394.

\bibitem{sch2} P.
Schenzel,  \emph{A criterion for I-adic completeness}, Arch. Math. (Basel) {\bf102} (2014),  25–33.



\bibitem{sv}R.Y.
Sharp and P. Vámos, \emph{
Baire's category theorem and prime avoidance in complete local rings},
Arch. Math. (Basel) {\bf{44}} (1985), no. 3, 243–248.



 \bibitem{st}J.
 Striuli, \emph{On extensions of modules}, J. Algebra {\bf285} (2005),  383–398.


 \bibitem{st2}J.
Striuli, \emph{Extensions of modules and uniform bounds of Artin-Rees type},
Thesis (Ph.D.)–University of Kansas. 2005.

\bibitem{v}
Wolmer V. Vasconcelos, \emph{Divisor theory in module categories}, North-Holland Mathematics Studies {\bf14} (North-Holland Publishing Co., Amsterdam, 1974).



\bibitem{ding} N.Q.
Ding and J.L. Chen, \emph{
The flat dimensions of injective modules}, 
Manuscripta Math. {\bf 78} (1993), no. 2, 165-177.


 \end{thebibliography}
\end{document}